\documentclass[reqno]{amsart}
\usepackage{amsmath,amsthm,amscd,amssymb,amsfonts, amsbsy}
\usepackage{latexsym, xcolor}
\usepackage{pxfonts}
\usepackage{mathrsfs}
\usepackage{bbm}
\usepackage{mathtools}
\usepackage{todonotes, marginnote}
\usepackage{enumitem}

\theoremstyle{plain}
\newtheorem{main}{Theorem}
\newtheorem{theorem}[equation]{Theorem}
\newtheorem{lemma}[equation]{Lemma}

\newtheorem{proposition}[equation]{Proposition}

\theoremstyle{definition}
\newtheorem{definition}[equation]{Definition}
\newtheorem{condition}[equation]{Condition}

\theoremstyle{remark}

\newcommand{\dv}{\operatorname{div}}
\newcommand{\capacity}{\operatorname{cap}}

\newcommand{\dist}{\operatorname{dist}}

\newcommand{\intr}{\operatorname{int}}

\newcommand{\tr}{\operatorname{tr}}

\numberwithin{equation}{section}

\newcommand{\bR}{\mathbb{R}}

\providecommand{\abs}[1]{\lvert#1\rvert}
\providecommand{\Abs}[1]{\left\lvert#1\right\rvert}

\providecommand{\norm}[1]{\lVert#1\rVert}

\renewcommand{\vec}[1]{\boldsymbol{#1}}

\newcommand{\newliminf}{\mathop{\mathrm{liminf}\vphantom{\mathrm{sup}}}}

\begin{document}
\title[Dirichlet problem for elliptic equations in double divergence form]
{Regular boundary points and the Dirichlet problem for elliptic equations in double divergence form}

\author[H. Dong]{Hongjie Dong}
\address[H. Dong]{Division of Applied Mathematics, Brown University,
182 George Street, Providence, RI 02912, United States of America}
\email{Hongjie\_Dong@brown.edu}
\thanks{H. Dong was partially supported by the NSF under agreement DMS-2350129.}

\author[D. Kim]{Dong-ha Kim}
\address[D. Kim]{Research Institute of Mathematics, Seoul National University, 1 Gwanak-ro, Gwanak-gu, Seoul 08826, Republic of Korea}
\email{kdh.pde@snu.ac.kr}

\author[S. Kim]{Seick Kim}
\address[S. Kim]{Department of Mathematics, Yonsei University, 50 Yonsei-ro, Seodaemun-gu, Seoul 03722, Republic of Korea}
\email{kimseick@yonsei.ac.kr}
\thanks{S. Kim is supported by the National Research Foundation of Korea (NRF) under agreement NRF-2022R1A2C1003322.
S. Kim is affiliated with the Intelligent Computational Science Institute (IN2CSI), Yonsei University, 50 Yonsei-Ro, Seodaemun-gu, Seoul 03722, Republic of Korea.}



\begin{abstract}
We study the Dirichlet problem for second-order elliptic operators in double divergence form, which arise as formal adjoints of non-divergence form operators and include the stationary Fokker--Planck--Kolmogorov equation.
Assuming that the leading coefficients have Dini mean oscillation and that the lower-order coefficients satisfy natural integrability conditions, we construct the Perron solution in arbitrary bounded domains.
We prove that a boundary point is regular with respect to the operator if and only if it satisfies the classical Wiener criterion for the Laplacian.
In particular, the Dirichlet problem is uniquely solvable in every bounded domain that is regular for the Laplacian.
\end{abstract}

\maketitle

\section{Introduction}
We consider the second-order elliptic operator in double divergence form,
\[
L^*u=\dv^2(\mathbf A u)-\dv(\boldsymbol b u)+cu = D_{ij}(a^{ij}u)-D_i(b^i u)+cu,
\]
defined on a domain $\Omega\subset\mathbb R^d$, $d\ge2$, where the summation convention is adopted.

The principal coefficient matrix $\mathbf A=(a^{ij})$ is assumed to be symmetric, uniformly elliptic, and of Dini mean oscillation.
In addition, the lower-order coefficients are assumed to satisfy
\[
\boldsymbol{b}=(b^1,\ldots,b^d)\in L^{p_0}(\Omega),
\qquad c\in L^{p_0/2}(\Omega), \qquad c\le0,
\]
for some $p_0>d$.
The precise assumptions are stated in Conditions \ref{cond1} and \ref{cond2}.

The operator $L^*$ is the formal adjoint of the second-order elliptic operator $L$ in non-divergence form, given by
\[
Lv=\tr(\mathbf{A} D^2v)+\boldsymbol{b} \cdot Dv+cv =a^{ij}D_{ij}v+b^iD_iv+cv.
\]
Thus the study of equations in double divergence form is naturally complementary to the study of equations in non-divergence form.
This is in sharp contrast with the divergence form setting: the formal adjoint of a divergence form operator is again of divergence form, whereas the formal adjoint of a non-divergence form operator is, in general, of double divergence form.
Hence the adjoint theory of non-divergence form equations leads outside the original class of operators, and it is precisely here that the parallel between divergence form and non-divergence form equations breaks down.

Such operators arise naturally in various areas of analysis and mathematical physics. 
A notable example is the stationary Fokker--Planck--Kolmogorov equation, which describes the steady-state distribution of diffusion processes.
For a comprehensive treatment of such equations, we refer the reader to \cite{BKRS2015}.

This paper investigates the Dirichlet problem associated with the operator $L^*$:
\begin{equation}	\label{eq_dirichlet}
L^*u=0 \quad \text{in } \Omega, \qquad u=g \quad \text{on } \partial\Omega.
\end{equation}
A function $u$ is said to be a solution of \eqref{eq_dirichlet} if it is a weak solution of $L^*u=0$ in $\Omega$, in the sense of Definition \ref{def01}, and attains the boundary data $g$ continuously on $\partial\Omega$.

Our main result shows that, under the above assumptions on the coefficients, the Dirichlet problem \eqref{eq_dirichlet} admits a unique solution $u\in C(\overline\Omega)$ whenever $\Omega$ is regular with respect to the Laplacian.
More generally, for an arbitrary bounded domain $\Omega$, we construct the Perron solution of \eqref{eq_dirichlet} and prove that a boundary point $x_0\in\partial\Omega$ is regular with respect to $L^*$, in the sense of Definition \ref{def_regpt}, if and only if it satisfies the Wiener criterion for the Laplacian.

The analysis of the Dirichlet problem \eqref{eq_dirichlet} presents several intrinsic difficulties. Notably, the operator $L^*$ does not satisfy the maximum principle.
For example, in the one-dimensional equation
\[
(a(x)u)''=0 \quad \text{in } (0,1),
\]
one solution is given by $u(x)=1/a(x)$, which may attain its maximum strictly inside the interval if $a$ has a strict interior minimum.
This departure from classical elliptic theory significantly complicates the qualitative analysis of solutions.
The same example also shows that solutions of $L^*u=0$ do not, in general, gain regularity beyond that of the leading coefficients, since $u=1/a$ is no smoother than the coefficient $a$ itself.
Another key difference from divergence form and non-divergence form equations is that constant functions are not, in general, solutions of $L^*u=0$, even in the absence of lower-order coefficients.
Consequently, the solution space is not invariant under the addition of constants.

These distinctive features necessitate the development of methods tailored to equations in double divergence form and, in particular, to the analysis of the Dirichlet problem \eqref{eq_dirichlet}.

The characterization of regular boundary points is a fundamental problem in the theory of elliptic equations.
For the Laplace equation, this problem was classically resolved by Wiener through a criterion based on capacity.
Subsequent work extended the theory to more general classes of elliptic equations.
In the case of divergence form equations with bounded measurable coefficients, Littman, Stampacchia, and Weinberger \cite{LSW63} proved that the regular boundary points are precisely those for the Laplacian.
This result was further generalized by Gr\"uter and Widman \cite{GW82}, who treated equations with possibly non-symmetric coefficients.

For elliptic equations in non-divergence form, the situation is more delicate.
Oleinik \cite{Oleinik} proved the equivalence of regular boundary points with those for the Laplacian under the assumption that the coefficients are of class $C^{3,\alpha}$.
Herv\'e \cite{Herve} extended this result to H\"older continuous coefficients, and Krylov \cite{Krylov67} further generalized it to Dini continuous coefficients.
However, Miller \cite{Miller} constructed a counterexample showing that the equivalence may fail even for uniformly continuous coefficients; see also \cite{Bauman84a, KY17}.
Bauman \cite{Bauman85} established a Wiener-type criterion for non-divergence form equations with continuous coefficients, although her result did not settle whether the regular boundary points for the operator coincide with those for the Laplacian.
Recently, the present authors resolved this question by proving the equivalence under the assumption that the leading coefficients have Dini mean oscillation; see \cite{DKK25a, DKK25b}.

Compared with the extensive literature on divergence form and non-divergence form equations, the Dirichlet problem for equations in double divergence form has received relatively limited attention.
Nevertheless, several notable contributions exist; see, for example, \cite{Sjogren73, Bauman84b, FS84, MMcO2007, EM2017}.

In \cite{Sjogren73}, Sj\"ogren studied the Dirichlet problem \eqref{eq_dirichlet} under the assumptions that the coefficients are locally H\"older continuous and that $\Omega$ is a bounded $C^{1,\alpha}$ domain.
This work was later extended in \cite{Sjogren75}, where he treated Dini continuous leading coefficients and locally bounded lower-order coefficients, assuming that $\Omega$ is of class $C^2$.
In addition, Sj\"ogren proved in \cite{Sjogren73} that weak solutions of $L^*u=0$ are locally H\"older continuous when the coefficients are locally H\"older continuous, and are continuous when the leading coefficients are Dini continuous.

More recently, it was shown in \cite{DK17, DEK18} that weak solutions of $L^*u=0$ are continuous under the weaker assumption that the leading coefficients have Dini mean oscillation.

The novelty of the present work lies in extending the classical theory of boundary regularity to elliptic equations in double divergence form under minimal continuity assumptions on the leading coefficients.
Since operators in double divergence form do not generally satisfy the maximum principle and do not preserve constants, the arguments used for divergence form and non-divergence form equations require substantial modifications.
We overcome these difficulties by first establishing a comparison principle that plays the role of the maximum principle, and then by introducing a tailored notion of capacity adapted to the double divergence structure.
We show that the resulting notion of boundary regularity is nevertheless governed by the classical Wiener criterion for the Laplacian.
To the best of our knowledge, the present work is the first to study the Dirichlet problem for elliptic equations in double divergence form at the level of generality considered here.

The paper is organized as follows.
In Section \ref{sec_main}, we introduce the assumptions on the coefficients and state the main results.
Section \ref{sec3} contains preliminary results, including the construction of the Perron solution in an arbitrary domain.
In the same section, we establish two-sided pointwise bounds for the Green function $G^*(x,y)$ associated with the operator $L^*$, which play a central role in our analysis.
These bounds are derived from earlier results \cite{HK20, DK21} on the Green function $G(x,y)$ for the operator $L$, together with the symmetry identity $G^*(x,y)=G(y,x)$.

In Section \ref{sec4}, we introduce the notions of potential and capacity adapted to the operator $L^*$, and prove a  theorem showing that the $L^*$-capacity is comparable to the corresponding capacity for the Laplacian.
Section \ref{sec5} contains the proof of one of our main results: in dimensions three and higher, a boundary point is regular with respect to $L^*$ if and only if the classical Wiener criterion is satisfied at that point.
Finally, Section \ref{sec_2d} is devoted to the two-dimensional analogue of this result.

\section{Main Results}			\label{sec_main}

We fix the following concentric open balls:
\begin{equation}			\label{eq_balls}
\mathcal{B}=B_{4R_0}(0),\quad \mathcal{B}'=B_{2R_0}(0), \quad \mathcal{B}''=B_{R_0}(0),
\end{equation}
where $R_0>0$ is chosen sufficiently large so that $\Omega\subset \mathcal{B}''$.
These balls will remain fixed throughout the paper.

\begin{condition}			\label{cond1}
The coefficients of $L$ are measurable and defined in the whole space $\bR^d$.
The principal coefficient matrix $\mathbf{A}=(a^{ij})$ is symmetric and satisfies the uniform ellipticity condition
\[
\lambda \abs{\xi}^2 \le \mathbf{A}(x) \xi  \cdot \xi \le \lambda^{-1} \abs{\xi}^2,\quad \forall x \in \bR^d,\;\;\forall \xi \in \bR^d,
\]
where $\lambda \in (0,1]$ is a constant.
The lower-order coefficients $\boldsymbol{b}=(b^1,\ldots,b^d)$ and $c$ satisfy
\[
\boldsymbol{b}\in L^{p_0}_{\mathrm{loc}}(\mathbb{R}^d), \qquad c\in L^{p_0/2}_{\mathrm{loc}}(\mathbb R^d)
\]
for some $p_0>d$, and
\[
\norm{\vec b}_{L^{p_0}(\mathcal{B})}+ \norm{c}_{L^{p_0/2}(\mathcal{B})} \le \Lambda,
\]
where $\Lambda=\Lambda(\mathcal B)<\infty$.
In addition, we assume that $c \le 0$.
\end{condition}

\begin{condition}			\label{cond2}
The mean oscillation function $\omega_{\mathbf A}: \bR_+ \to \bR$ defined by
\[
\omega_{\mathbf A}(r):=\sup_{x\in \mathcal{B}} \fint_{\mathcal{B} \cap B_r(x)} \,\abs{\mathbf A(y)-\bar {\mathbf A}_{x,r}}\,dy, \quad \text{where }\;\bar{\mathbf A}_{x,r} :=\fint_{\mathcal{B} \cap B_r(x)} \mathbf A,
\]
satisfies the Dini condition
\[
\int_0^1 \frac{\omega_{\mathbf A}(t)}t \,dt <+\infty.
\]
\end{condition}
We note that Condition~\ref{cond2} implies that $\mathbf A$ has a continuous representative; see \cite[Appendix]{HK20}.

The following statements summarize the main results of the paper. 
Although they are repeated from the main text, we include them here for the reader's convenience.

\begin{main}[Theorems \ref{thm0800sat} and \ref{thm1510mon}]
Let $\Omega$ be a bounded domain in $\mathbb R^d$, $d\ge2$.
Assume that Conditions \ref{cond1} and \ref{cond2} hold.
Then a point $x_0\in\partial\Omega$ is regular with respect to $L^*$, in the sense of Definition \ref{def_regpt}, if and only if it is regular with respect to the Laplacian.
\end{main}

\begin{main}[Theorem \ref{thm0802sat} and \ref{thm1515mon}]
Let $\Omega$ be a bounded regular domain in $\mathbb R^d$, $d\ge2$, in the sense of Definition \ref{def_reg}.
Assume that Conditions \ref{cond1} and \ref{cond2} hold.
Then, for every $f\in C(\partial\Omega)$, the Dirichlet problem
\[
L^*u=0 \;\text{ in }\;\Omega, \quad u=f \;\text{ on }\;\partial\Omega
\]
has a unique solution $u\in C(\overline\Omega)$.
\end{main}

\section{Preliminary}			\label{sec3}
\subsection{Some preliminary results}

We record several results concerning solutions of the equation $L^*u=0$ under Conditions \ref{cond1} and \ref{cond2}.

\begin{definition}		\label{def01}
We say that $u\in L^1_{\rm loc}(\Omega)$ is a weak solution of
\[
L^*u=0\quad\text{in }\;\Omega
\]
if $\boldsymbol b u\in L^1_{\rm loc}(\Omega)$, $cu\in L^1_{\rm loc}(\Omega)$, and
\[
\int_\Omega u \left(a^{ij}D_{ij}\eta + b^i D_i \eta + c\eta \right) dx =0
\]
for every test function $\eta\in C_c^\infty(\Omega)$.
Let $\Omega$ be a bounded $C^{1,1}$ domain, and let $g\in L^1(\partial\Omega)$.
We say that $u\in L^1(\Omega)$ is a weak solution of
\[
L^*u=0 \quad \text{in }\; \Omega,\qquad u=g \quad \text{on }\; \partial\Omega,
\]
if $\boldsymbol b u\in L^1(\Omega)$, $cu\in L^1(\Omega)$, and
\begin{equation} \label{eq0958sat}
\int_\Omega u\left(a^{ij}D_{ij}\eta+b^iD_i\eta+c\eta\right) dx=
\int_{\partial\Omega} g\,a^{ij}D_j\eta\,\nu_i\,dS
\end{equation}
for every $\eta\in C^\infty(\overline\Omega)$ satisfying $\eta=0$ on $\partial\Omega$.
Here $\nu=(\nu_1,\ldots,\nu_d)$ denotes the unit outer normal vector to $\partial\Omega$.
\end{definition}

The following theorem establishes solvability of the Dirichlet problem \eqref{eq_dirichlet} in bounded $C^{1,1}$ domains.

\begin{theorem}	\label{thm01}
Assume that Conditions \ref{cond1} and \ref{cond2} hold.
Let $\Omega$ be a bounded $C^{1,1}$ domain, and let $g\in C(\partial\Omega)$.
Then the Dirichlet problem \eqref{eq_dirichlet} has a unique weak solution $u\in C(\overline\Omega)$.
Moreover, if $g \ge 0$ on $\partial\Omega$, then $u\ge0$ in $\Omega$.
\end{theorem}

\begin{proof}
The existence and uniqueness of a weak solution $u\in L^{q_0}(\Omega)$, where
\[
1/q_0+2/p_0=1,
\]
follow from the transposition, or duality, method.
This approach relies on the unique solvability of the Dirichlet problem
\begin{equation}			\label{eq1649sun}
L v=\varphi \;\text{ in }\;\Omega,\qquad v=0\;\text{ on }\;\partial\Omega,
\end{equation}
for each $\varphi\in L^{p_0/2}(\Omega)$, with solutions sought in the space $\mathring W^{2,p_0/2}(\Omega)$.
Here $\mathring W^{2,p_0/2}(\Omega)$ denotes the closure, in the $W^{2,p_0/2}$-norm, of the set
\[
\{v \in C^\infty(\overline \Omega): v=0 \;\text{ on }\;\partial\Omega\}.
\]

Define the linear functional $T:L^{p_0/2}(\Omega)\to\mathbb R$ by
\[
T(\varphi)=\int_{\partial\Omega} g \,a^{ij}D_jv\,\nu_i\,dS,
\]
where $v\in \mathring W^{2,p_0/2}(\Omega)$ is the unique solution of \eqref{eq1649sun}.
The existence and uniqueness of such $v$ are guaranteed by \cite[Theorem 2.8]{Krylov2023e}.
By the trace theorem and the $W^{2,p}$-estimate for \eqref{eq1649sun}, we have
\begin{align*}
\Abs{\int_{\partial\Omega} g \,a^{ij}D_jv \,\nu_i\,dS}& \le C \norm{g}_{L^{q_0}(\partial\Omega)} \norm{Dv}_{L^{p_0/2}(\partial\Omega)} \le C  \norm{g}_{L^{q_0}(\partial\Omega)} \norm{Dv}_{W^{1,p_0/2}(\Omega)} \\
&\le C  \norm{g}_{L^{q_0}(\partial\Omega)} \norm{v}_{W^{2,p_0/2}(\Omega)}\le  C  \norm{g}_{L^\infty(\partial\Omega)}  \norm{\varphi}_{L^{p_0/2}(\Omega)}.
\end{align*}
Hence, $T$ is a bounded functional on $L^{p_0/2}(\Omega)$.

By the Riesz representation theorem, there exists a unique $u\in L^{q_0}(\Omega)$ such that
\[
T(\varphi)=\int_\Omega u \varphi,\qquad  \varphi \in L^{p_0/2}(\Omega).
\]
Since $u\in L^{q_0}(\Omega)$, the assumptions on the lower-order coefficients imply that $\boldsymbol b u\in L^1(\Omega)$ and $cu\in L^1(\Omega)$.

For any $\eta\in C^\infty(\overline\Omega)$ with $\eta=0$ on $\partial\Omega$, taking $\varphi=L\eta$ in the definition of $T$ gives the identity \eqref{eq0958sat}.
Furthermore, by \cite[Theorem 1.8]{DEK18}, we have $u\in C(\overline\Omega)$; see also \cite[Theorem 2.9]{Kim2023}.
Thus $u\in C(\overline\Omega)$ is the unique weak solution of \eqref{eq_dirichlet}.

To prove that $u\ge0$ in $\Omega$ when $g\ge0$ on $\partial\Omega$, let $\varphi\in C_c^\infty(\Omega)$ with $\varphi\ge0$, and let $v\in\mathring W^{2,p_0/2}(\Omega)$ be the solution of \eqref{eq1649sun}.
Then
\[
\int_\Omega u\varphi\,dx = \int_{\partial\Omega} g\,a^{ij}D_jv\,\nu_i\,dS.
\]
Since $v=0$ on $\partial\Omega$, we have $Dv=(Dv\cdot\nu)\nu$ on $\partial\Omega$. Hence
\begin{equation}	\label{eq1056sun}
\int_\Omega u\varphi\,dx
=\int_{\partial\Omega} g\,(Dv\cdot\nu)\,a^{ij}\nu_i\nu_j\,dS.
\end{equation}
Because $Lv=\varphi\ge0$ in $\Omega$ and $v=0$ on $\partial\Omega$, the maximum principle gives $v\le0$ in $\Omega$.
Consequently, $Dv\cdot\nu\ge0$ on $\partial\Omega$.

Since $g\ge0$ on $\partial\Omega$ and $a^{ij}\nu_i\nu_j\ge \lambda>0$ on $\partial\Omega$, the right-hand side of \eqref{eq1056sun} is nonnegative.
Hence
\[
\int_\Omega u\varphi\,dx\ge0
\]
for every $\varphi\in C_c^\infty(\Omega)$ with $\varphi\ge0$.
It follows that $u\ge0$ in $\Omega$.
\end{proof}

The following result is the Harnack inequality for nonnegative solutions of $L^*u=0$.
For the proof, we refer the reader to \cite[Theorem 4.3]{GK24}.
We also note that a Harnack inequality was previously established in \cite{Mamedov} under the stronger assumption that $\mathbf A$ is Dini continuous.

\begin{theorem}[Harnack inequality]		\label{thm02}
Assume that Conditions \ref{cond1} and \ref{cond2} hold.
Let $u$ be a nonnegative weak solution of $L^*u=0$ in $B_{4r}=B_{4r}(x_0)\subset\mathcal B$. Then
\[
\sup_{B_r} u \le N \inf_{B_r}u,
\]
where the constant $N$ depends only on $d$, $\lambda$, $\Lambda$, $\omega_{\mathbf A}$, $p_0$, and $R_0$.
\end{theorem}

\begin{lemma}	\label{lem01}
Assume that Conditions \ref{cond1} and \ref{cond2} hold.
Let $u\ge0$ be a weak solution of $L^*u=0$ in a bounded domain $\Omega$.
Then either $u>0$ in $\Omega$ or $u\equiv0$ in $\Omega$.
\end{lemma}

\begin{proof}
Suppose that $u\not\equiv0$ in $\Omega$.
Then there exists $x_0\in\Omega$ such that $u(x_0)>0$.
By the Harnack inequality, Theorem~\ref{thm02}, it follows that $u>0$ in a neighborhood of $x_0$.
Applying the Harnack inequality iteratively along a chain of overlapping balls contained in $\Omega$, we conclude that $u>0$ throughout $\Omega$.
\end{proof}

\begin{lemma}	\label{lem02}
Assume that Conditions \ref{cond1} and \ref{cond2} hold.
Then there exists a function $W\in C(\overline{\mathcal B})$ such that $W$ is a weak solution of $L^*W=0$ in $\mathcal B$ and
\begin{equation}	\label{eq1532tue}
\gamma_0 \le W \le \gamma_0^{-1}
\quad \text{in } \mathcal B,
\end{equation}
where $\gamma_0\in(0,1]$ depends only on $d$, $\lambda$, $\Lambda$, $\omega_{\mathbf A}$, $p_0$, and $R_0$.
\end{lemma}

\begin{proof}
Let $W\in C(\overline{\mathcal B})$ be the solution of the boundary value problem
\[
L^*W=0 \quad \text{in }\; \mathcal B,
\qquad W=1 \quad \text{on }\; \partial\mathcal B.
\]
The existence of such a solution follows from Theorem \ref{thm01}, which also gives $W\ge0$ in $\mathcal B$.
By Lemma \ref{lem01}, we have $W>0$ in $\mathcal B$.
Since $W$ is continuous on the compact set $\overline{\mathcal B}$, it attains a positive minimum and a finite maximum on $\overline{\mathcal B}$.
Therefore, there exists a constant $\gamma_0 \in (0,1]$, depending only on the specified parameters, such that inequality \eqref{eq1532tue} holds. 
\end{proof}

\begin{lemma} \label{lem03}
Assume that Conditions \ref{cond1} and \ref{cond2} hold.
Let $W$ be the function given in Lemma \ref{lem02}, and let $\Omega\subset\mathcal B$ be a domain.
If $u\in C(\overline\Omega)$ is a weak solution of $L^*u=0$ in $\Omega$, then neither the minimum nor the maximum of $u/W$ is attained in $\Omega$, unless $u/W$ is constant.
\end{lemma}

\begin{proof}
Set
\[
v=u-mW, \qquad m:=\min_{\overline\Omega}\,\frac{u}{W}.
\]
Since $W>0$ on $\overline\Omega$, we have $v\ge0$ in $\overline\Omega$.
Moreover, $v\in C(\overline\Omega)$ and $L^*v=0$ in $\Omega$.

Suppose that the minimum of $u/W$ is attained at some point $x_0\in\Omega$.
Then
\[
\frac{u(x_0)}{W(x_0)}=m,
\]
and hence $v(x_0)=0$.
By Lemma \ref{lem01}, it follows that $v\equiv0$ in $\Omega$, that is, $u\equiv mW$ in $\Omega$.
Thus $u/W$ is constant in $\Omega$.
The assertion for the maximum follows by applying the same argument to $-u$.
\end{proof}
 
\begin{lemma}		\label{lem04}
Assume that Conditions \ref{cond1} and \ref{cond2} hold.
Let $\Omega\subset\mathcal B$ be a bounded domain, and let $u\in C(\overline\Omega)$ be a weak solution of $L^*u=0$ in $\Omega$.
Then
\begin{equation}	\label{eq0757mon}
\max_{\overline\Omega}\, \abs{u} \le N \max_{\partial\Omega}\, \abs{u},
\end{equation}
where $N>0$ depends only on $d$, $\lambda$, $\Lambda$, $\omega_{\mathbf A}$, $p_0$, and $R_0$.
\end{lemma}
 
 \begin{proof}
By \eqref{eq1532tue} and Lemma \ref{lem03}, for every $x\in\overline\Omega$ we have
\[
u(x) \le W(x)\,\max_{\overline\Omega}\,\frac{u}{W}
=W(x)\,\max_{\partial\Omega}\,\frac{u}{W} \le \gamma_0^{-2}\max_{\partial\Omega}\, \abs{u}.
\]
Applying the same estimate to $-u$, we obtain \eqref{eq0757mon} with $N=\gamma_0^{-2}$.
\end{proof}

\subsection{Supersolutions,  subsolutions, and Perron method}
Let $B$ be a ball contained in $\mathcal B$.
For each $f\in C(\partial B)$, let $u\in C(\overline B)$ denote the solution of the Dirichlet problem
\[
L^*u=0 \quad \text{in } \; B,
\qquad u=f \quad \text{on }\; \partial B.
\]
The existence and uniqueness of such a solution follow from Theorem \ref{thm01}.

Fix $x\in B$.
The map $f\mapsto u(x)$ defines a linear functional on $C(\partial B)$.
Moreover, by Lemma \ref{lem04},
\[
\abs{u(x)} \le \max_{\overline B}\, \abs{u} \le N \max_{\partial B}\, \abs{f},
\]
and hence this functional is bounded.
It is also positive by Theorem \ref{thm01}.

By the Riesz representation theorem, there exists a finite positive Radon measure $\omega_B^x$ on $\partial B$ such that
\[
u(x)=\int_{\partial B} f\,d\omega_B^x \qquad
\text{for all } f\in C(\partial B).
\]
We call $\omega_B^x$ the $L^*$-harmonic measure at $x$ in $B$.

The existence of the $L^*$-harmonic measure for every ball $B\subset\mathcal B$ allows us to define $L^*$-supersolutions, $L^*$-subsolutions, and $L^*$-solutions as follows.

\begin{definition}	\label{def0617wed}
Let $\mathcal{D} \subset \mathcal{B}$ be an open set.
An extended real-valued function $u$ is called an $L^*$-supersolution in $\mathcal{D}$ if the following conditions hold:
\begin{enumerate}
\item[(i)]
$u$ is not identically $+\infty$ in any connected component of $\mathcal{D}$.
\item[(ii)]
$u>-\infty$ in $\mathcal{D}$.
\item[(iii)]
$u$ is lower semicontinuous in $\mathcal{D}$.
\item[(iv)]
For every ball $B\Subset\mathcal D$ and every $x\in B$,
\[
u(x) \ge \int_{\partial B} u\,d\omega^x_{B}.
\]
\end{enumerate}

An $L^*$-subsolution in $\mathcal D$ is defined analogously: $u$ is an $L^*$-subsolution if and only if $-u$ is an $L^*$-supersolution.

If $u$ is both an $L^*$-supersolution and an $L^*$-subsolution in $\mathcal D$, then we call $u$ an $L^*$-solution in $\mathcal D$.
\end{definition}

Note that an $L^*$-solution in $\mathcal D$ is necessarily finite and continuous in $\mathcal D$ by definition.

\begin{lemma}
A function $u$ is an $L^*$-solution in $\mathcal D$ if and only if $u$ is a continuous weak solution of $L^*u=0$ in $\mathcal D$.
\end{lemma}

\begin{proof}
Suppose first that $u$ is an $L^*$-solution in $\mathcal D$.
Then $u$ is continuous in $\mathcal{D}$.
Moreover, for every ball $B\Subset\mathcal D$ and every $x\in B$,
\begin{equation}			\label{eq0925thu}
u(x)=\int_{\partial B} u \,d\omega_{B}^{x}.
\end{equation}
Therefore, by Theorem \ref{thm01}, $u$ coincides in $B$ with the unique weak solution of $L^*u=0$ having boundary data $u\vert_{\partial B}$.
Since $B\Subset\mathcal D$ is arbitrary, it follows by a standard localization argument that $u$ is a weak solution of $L^*u=0$ in $\mathcal D$.

Conversely, suppose that $u$ is a continuous weak solution of $L^* u=0$ in $\mathcal D$.
Let $B\Subset\mathcal D$ be a ball.
By uniqueness of the Dirichlet problem in $B$, the identity \eqref{eq0925thu} holds for every $x\in B$.
Therefore $u$ is an $L^*$-solution in $\mathcal D$.
\end{proof}

The following lemma is the analogue of \cite[Lemma 3.21]{DKK25a}, and its proof is essentially unchanged.
We note, however, that the strong minimum principle for $L^*$ takes a different form.

\begin{lemma}			\label{lem2334sun}
Let $u\in\mathfrak S^+(\mathcal D)$, where $\mathfrak S^+(\mathcal D)$ denotes the set of all $L^*$-supersolutions in $\mathcal D$.
Then the following statements hold.
\begin{enumerate}
\item  (Strong minimum principle)
Assume that $\mathcal D$ is connected.
If $u\in\mathfrak S^+(\mathcal D)$ is nonnegative and $u(x_0)=0$ for some $x_0\in\mathcal D$, then $u\equiv0$ in $\mathcal D$.
\item (Comparison principle)
If $u \in \mathfrak{S}^+(\mathcal{D})$ and $\liminf_{y \to x,\; y\in \mathcal{D}} u(y) \ge 0$ for every $x \in \partial \mathcal{D}$, then $u \ge 0$ in $\mathcal{D}$.
\item
If $u\in\mathfrak S^+(\mathcal D)$, then for every $x_0\in\mathcal D$, $\liminf_{x \to x_0} u(x)=u(x_0)$.
\item
If $u,v\in\mathfrak S^+(\mathcal D)$ and $c>0$, then $cu$, $u+v$, $\min(u,v) \in \mathfrak{S}^+(\mathcal{D})$.
\item
(Pasting lemma)
Let $\mathcal D'\subset\mathcal D$ be open, and let $v\in\mathfrak S^+(\mathcal D')$.
Define
\[
w= \begin{cases}
\min(u,v)& \text{in }\;\mathcal{D}',\\
\phantom{\min}u &  \text{in }\; \mathcal{D} \setminus \mathcal{D}'.
\end{cases}
\]
If $w$ is lower semicontinuous on $\partial\mathcal D'$, then $w\in\mathfrak S^+(\mathcal D)$.
\end{enumerate}
\end{lemma}
\begin{proof}
\begin{enumerate}[leftmargin=*]
\item 
Let $\mathcal U=\{x \in \mathcal{D}: u(x)=0\}$.
Since $u\ge0$ and $u$ is lower semicontinuous, the set $\mathcal U$ is relatively closed in $\mathcal D$. We claim that $\mathcal U$ is also open.
Once this is proved, the connectedness of $\mathcal D$ implies that $\mathcal  U=\mathcal D$, and the proof is complete.

To prove the claim, let $y\in \mathcal U$.
Choose $r>0$ such that $B_r(y)\subset\mathcal D$.
Suppose, for contradiction, that there exists $y'\in B_r(y)\setminus \mathcal U$, and set $\rho=\abs{y'-y}<r$.
Then $y'\in\partial B_\rho(y)$ and $u(y')>0$.
Since $u$ is lower semicontinuous, we can choose a continuous function $f$ on $\partial B_\rho(y)$ such that $0\le f \le u$ on $\partial B_\rho(y)$ and $f(y')>0$.
Define
\[
v(x):=\int_{\partial B_\rho(y)} f\,d\omega_{B_\rho(y)}^x,
\qquad x\in B_\rho(y).
\]
Then, $v$ is a nonnegative $L^*$-solution in $B_\rho(y)$, and by Lemma~\ref{lem01},  $v(y)>0$.

On the other hand, using $0\le f\le u$ on $\partial B_\rho(y)$ and the defining inequality for the supersolution $u$, we obtain
\[
0<v(y) = \int_{\partial B_\rho(y)} f\,d\omega_{B_\rho(y)}^y
\le \int_{\partial B_\rho(y)} u\,d\omega_{B_\rho(y)}^y \le u(y)=0,
\]
a contradiction.
Therefore no such point $y'$ exists, and hence $B_r(y)\subset \mathcal U$.
Thus $\mathcal U$ is open.

\item
Suppose, for contradiction, that
\[
\liminf_{y \to x,\; y\in \mathcal{D}} u(y) \ge 0\quad \text{for every }\; x \in \partial \mathcal{D},
\]
but that $u(y_0)<0$ for some $y_0\in\mathcal D$.
Let $W>0$ be the function given in Lemma~\ref{lem02}, and set
\[
m:= \inf_{\mathcal{D}}\left(\frac{u}{W}\right).
\]
Then $m<0$ and this infimum is attained at some point $z_0 \in \mathcal D$.
Define
\[
v:=u-mW.
\]
Then $v\in\mathfrak S^+(\mathcal D)$, $v\ge0$ in $\mathcal D$, and $v(z_0)=0$.
By the strong minimum principle, it follows that $v\equiv0$ in $\mathcal D$.
Since $m<0$ and $W$ is continuous and strictly positive on $\overline{\mathcal D}$, for every $x\in\partial\mathcal D$ we have
\[
\liminf_{y \to x,\; y\in \mathcal{D}} u(y) = \liminf_{y \to x,\; y\in \mathcal{D}} mW(y) = mW(x)<0,
\]
contrary to the assumption.
Therefore $u\ge0$ in $\mathcal D$.

\item
Since $u$ is lower semicontinuous, $u(x_0)\le \liminf_{x\to x_0}u(x)$.
If the inequality were strict, choose $\alpha \in\mathbb R$ such that
\[
u(x_0)<\alpha W(x_0)<\liminf_{x\to x_0}u(x),
\]
where $W$ is in Lemma~\ref{lem02}.
By continuity of $W$, for sufficiently small $r>0$,
\[
u > \alpha W \quad\text{on }\;\partial B_r(x_0).
\]
Since $u-\alpha W$ is an $L^*$-supersolution, the comparison principle yields $u \ge \alpha W$ in $B_r(x_0)$, contradicting $u(x_0)<\alpha W(x_0)$.
Therefore $\liminf_{x\to x_0}u(x)=u(x_0)$.

\item
This follows directly from the definition of an $L^*$-supersolution and the positivity of the $L^*$-harmonic measure.

\item
It suffices to verify the mean value inequality for $w$.
Let $B\Subset\mathcal D$, and let $f\in C(\partial B)$ satisfy $f\le w$. Set
\[
h(x)=\int_{\partial B}f\,d\omega_B^x.
\]
Since $f\le w\le u$ on $\partial B$, the comparison principle gives $h\le u$ in $B$.

We claim that $h\le v$ in $B\cap\mathcal D'$.
Let $\mathcal U$ be a connected component of $B\cap\mathcal D'$.
If $x\in B\cap\partial\mathcal D'$, then
\[
\limsup_{y\to x,\, y \in \mathcal U} h(y) = h(x) \le u(x) = w(x)
\le \liminf_{y\to x,\, y \in \mathcal U} w(y) \le \liminf_{y\to x,\, y \in \mathcal U} v(y).
\]
If $x\in\partial B \cap \overline{\mathcal D'}$, then
\[
\limsup_{y\to x,\, y \in \mathcal U} h(y) = f(x) \le w(x)
\le \liminf_{y\to x,\, y \in \mathcal U} w(y) \le \liminf_{y\to x,\, y \in \mathcal U} v(y).
\]
Hence the comparison principle gives $h\le v$ in $\mathcal U$, and therefore in $B \cap \mathcal D'$.
Together with $h\le u$ in $B$, this yields $h\le w$ in $B$.
Taking the supremum over all such $f$, using the lower semicontinuity of $w$, gives
\[
\int_{\partial B}w\,d\omega_B^x\le w(x). \qedhere
\]
\end{enumerate}
\end{proof}

\begin{definition}
Let $u\in\mathfrak S^+(\mathcal D)$ and let $B\Subset\mathcal D$ be a ball.
We define the $L^*$-harmonic lowering of $u$ over $B$, denoted by $\mathscr{E}_{B}u$, by
\[
\mathscr{E}_{B}u(x) = \begin{cases}
\phantom{\int_{\partial B}}u(x)& \text{in }\;\mathcal{D} \setminus  B,\\
\int_{\partial B} u\,d\omega^x_{B} &  \text{in }\; B.
\end{cases}
\]
\end{definition}

\begin{lemma}
The following properties hold:
\begin{enumerate}
\item
$\mathscr{E}_{B}u \le u$.
\item
$\mathscr{E}_{B}u$ is an $L^*$-supersolution in $\mathcal{D}$.
\item
$\mathscr{E}_{B}u$ is an $L^*$-solution in $B$.
\end{enumerate}
\end{lemma}
\begin{proof}
The proof is the same as that of \cite[Lemma 3.23]{DKK25a}.
\end{proof}

\begin{definition}
A collection $\mathscr F$ of $L^*$-supersolutions in $\mathcal D$ is said to be saturated in $\mathcal D$ if the following conditions hold:
\begin{enumerate}
\item[(i)]
If $u$, $v \in \mathscr F$, then $\min(u,v) \in \mathscr F$.
\item[(ii)]
If $u\in\mathscr F$ and $B\Subset\mathcal D$ is a ball, then $\mathscr E_Bu\in\mathscr F$.
\end{enumerate}
\end{definition}

\begin{lemma}			\label{lem1034thu}
Let $\mathscr{F}$ be a saturated collection of $L^*$-supersolutions in $\mathcal{D}$, and define
\[
v(x) := \inf_{u \in \mathscr{F}} u(x).
\]
If $v$ does not take the value $-\infty$ in $\mathcal D$, then $v$ is an $L^*$-solution in $\mathcal D$.
\end{lemma}
\begin{proof}
The proof is the same as that of \cite[Lemma 3.25]{DKK25a}.
\end{proof}

\begin{definition}			\label{def_perron}
Let $f$ be a function on $\partial\mathcal D$.
Consider the Dirichlet problem
\[
L^*u=0 \;\text{ in }\; \mathcal{D},\quad
u=f\; \text{ on }\;\partial \mathcal{D}.
\]
The upper Perron solution $\overline H_f$ is defined by
\[
\overline H_f(x)=\inf \left\{v(x): v \in \mathfrak{S}^+(\mathcal{D}), \;\newliminf_{y\to x_0,\, y \in \mathcal{D}}\, v(y) \ge f(x_0)\;\text{ for all }\;x_0 \in\partial\mathcal{D}\right\}.
\]
Similarly, the lower Perron solution $\underline H_f$ is defined by
\[
\underline H_f(x)=\sup \left\{v(x): v \in \mathfrak{S}^-(\mathcal{D}), \;\limsup_{y\to x_0,\, y \in \mathcal{D}}\, v(y) \le f(x_0)\;\text{ for all }\;x_0 \in\partial\mathcal{D}\right\},
\]
where $ \mathfrak{S}^-(\mathcal{D})$ denotes the set of all $L^*$-subsolution in $\mathcal{D}$.
\end{definition}

It is clear that $\underline H_f \le \overline H_f$. 
Moreover,
\begin{equation}			\label{eq1208thu}
\underline H_f=-\overline H_{-f}.
\end{equation}

\begin{lemma}
Let $f\in C(\partial\mathcal D)$.
Then the upper and lower Perron solutions $\overline H_f$ and $\underline H_f$ are both $L^*$-solutions in $\mathcal D$.
\end{lemma}

\begin{proof}
Let $W$ be the function given in Lemma \ref{lem02}, and set
\[
w:= \gamma_0^{-1}\left(\max_{\partial \mathcal{D}}\,\abs{f}\right) W.
\]
Then $w \in \mathscr{F}$, where
\[
\mathscr F:=\left\{v \in \mathfrak{S}^+(\mathcal{D}): \liminf_{y\to x_0,\, y \in \mathcal{D}}\, v(y) \ge f(x_0)\;\text{ for all }\;x_0 \in\partial\mathcal{D}\right\}.
\]
It follows that $w \ge \overline H_f$.
Moreover, by the comparison principle, $v \ge -w$ for every $v \in \mathscr{F}$.
Therefore,
\[
-w \le \overline{H}_f \le w.
\]
The class $\mathscr F$ is saturated.
Therefore, by Lemma \ref{lem1034thu}, $\overline H_f$ is an $L^*$-solution in $\mathcal D$. 
Using the identity \eqref{eq1208thu}, we conclude that $\underline H_f$ is also an $L^*$-solution in $\mathcal D$.
\end{proof}

\begin{definition}                \label{def_regpt}
A point $x_0\in\partial\mathcal D$ is called regular if, for every $f\in C(\partial\mathcal D)$,
\[
\lim_{x \to x_0, \, x\in \mathcal D}\, \underline H_f(x)=\lim_{x \to x_0, \, x\in \mathcal D}\, \overline H_f(x)=f(x_0).
\]
\end{definition}

\subsection{Green function in a ball}			
In \cite[Theorem 4.1]{DKK25a} and \cite[Theorem 3.3]{DKK25b}, we constructed the Green function $G(x,y)$ for the operator $L$ in a ball in $\mathbb{R}^d$ for $d\ge 3$ and $d=2$, respectively.
In both cases, it was shown that
\begin{equation}			\label{symmetry}
G^*(x,y)=G(y,x)
\end{equation}
is the Green function for the adjoint operator $L^*$ in the same ball.
The following theorems are immediate consequences of \eqref{symmetry}, together with \cite[Theorem 4.13]{DKK25a} in the case $d\ge3$ and \cite[Theorem 3.15]{DKK25b} in the case $d=2$.

\medskip
Recall that $\mathcal{B}$ and $\mathcal{B}'$ are concentric balls in $\mathbb{R}^d$ as defined in \eqref{eq_balls}.

\begin{theorem}[$d \ge 3$]		\label{thm_green_function}
Assume that Conditions \ref{cond1} and \ref{cond2} hold.
Let $G^*(x, y)$ be the Green function for the operator $L^*$ in $\mathcal{B} \subset \mathbb{R}^d$, where $d\ge 3$.
There exists a constant $N_0>0$, depending only on $d$, $\lambda$, $\Lambda$, $\omega_{\mathbf A}$, and $R_0$, such that
\[
N_0^{-1} \abs{x-y}^{2-d} \le G^*(x,y)\le N_0 \abs{x-y}^{2-d}
\]
for all $x,y\in\overline{\mathcal B'}$ with $x\ne y$.
\end{theorem}

\begin{theorem}[$d=2$]	\label{thm_green_function2d}
Assume that Conditions \ref{cond1} and \ref{cond2} hold.
Let $B_{4r}=B_{4r}(x_0) \subset \mathcal{B} \subset \mathbb{R}^2$, and let $G^*(x, y)$ be the Green function for $L^*$ in $B_{4r}$.
Then there exists a constant $C_0>1$, depending only on $\lambda$, $\Lambda$, $\omega_{\mathbf A}$, and $R_0$, such that
\begin{equation}		\label{estimate_green}
\frac{1}{C_0} \log \left(\frac{3r}{\abs{x-y}}\right) \le G^*(x,y)\le C_0 \log \left(\frac{3r}{\abs{x-y}}\right),\qquad x, y\in\overline B_r,\quad x\ne y.
\end{equation}
\end{theorem}

\section{Potential and Capacity}		\label{sec4}
We note that, in what follows, the function $W$ introduced in Lemma \ref{lem02} plays the role of the constant function $1$, in analogy with the corresponding definitions and results in \cite{DKK25a}. 
In analogy with the notion of potential for $L$ used in \cite{DKK25a}, we call a nonnegative $L^*$-supersolution in $\mathcal B$ an $L^*$-potential, or simply a potential, if it is finite at every point of $\mathcal B$ and vanishes continuously on $\partial\mathcal B$.

The following definition is the analogue of \cite[Definition 5.5]{DKK25a}, with the function $W$ replacing the constant function $1$.

\begin{definition} \label{def_rfn}
For $E\Subset\mathcal B$, define
\[
u_E(x):= \inf\left\{ v(x): v\in\mathfrak S^+(\mathcal B),\ v\ge0 \ \text{in } \mathcal B,\ v\ge W \ \text{on } E \right\}, \qquad x\in\mathcal B.
\]
The lower semicontinuous regularization
\[
\hat u_E(x)=\sup_{r>0} \left(\inf_{\mathcal B\cap B_r(x)} u_E\right)
\]
is called the $L^*$-capacitary potential, or simply the capacitary potential, of $E$.
\end{definition}

\begin{lemma}				\label{lem1013sat}
The following properties hold:
\begin{enumerate}[leftmargin=*]
\item
$0\le \hat u_E \le u_E \le W$.
\item
$u_E=W$ on $E$, and $\hat u_E =W$ in $\intr (E)$.
\item
$\hat u_E$ is an $L^*$-potential.
\item
$u_E=\hat u_E$ in $\mathcal{B}\setminus \overline{E}$, and
$L^* u_E=L^* \hat u_E=0$ in $\mathcal{B}\setminus \overline{E}$.
\item
If $x_0\in \partial E$, then
\[
\liminf_{x\to x_0} \hat u_{E}(x) =\liminf_{x\to x_0,\; x\in \mathcal{B}\setminus  E} \hat u_E(x).
\]
\end{enumerate}
\end{lemma}
\begin{proof}
The proof is the same as that of \cite[Lemma 5.6]{DKK25a}, with the function $W$ replacing the constant function $1$.
\end{proof}

The following lemma is the analogue of \cite[Lemma 5.7]{DKK25a}.
Its proof is unchanged.

\begin{lemma}			\label{lem0800sun}
Let $u$ be a nonnegative $L^*$-supersolution in $\mathcal B$, and assume that $u<+\infty$ in $\mathcal B$.
For an open set $\mathcal D\subset\mathcal B$, define
\[
\mathscr{F}_{\mathcal D}=\{v \in \mathfrak{S}^+(\mathcal{B}):  v\ge 0\,\text{ in }\,\mathcal{B},\; v - u \in \mathfrak{S}^+(\mathcal{D})\}.
\]
For $x\in\mathcal B$, define
\[
\mathbb{P}_\mathcal{D} u(x)=\inf_{v \in \mathscr{F}_{\mathcal D}} v(x).
\]
Then the following properties hold:
\begin{enumerate}[leftmargin=*]
\item
$\mathbb{P}_\mathcal{D} u \in \mathfrak{S}^+(\mathcal{B})$ and $0 \le \mathbb{P}_\mathcal{D} u \le u$ in $\mathcal{B}$.
\item
$\mathbb{P}_\mathcal{D}u \in \mathscr{F}_{\mathcal D}$;  equivalently, $\mathbb{P}_\mathcal{D} u - u$ is an $L^*$-supersolution in $\mathcal{D}$.
\item
$\mathbb{P}_\mathcal{D} u$ is an $L^*$-solution in $\mathcal{B}\setminus \overline{\mathcal{D}}$.
\item
If $v \in \mathscr{F}_{\mathcal{D}}$, then $v-\mathbb{P}_\mathcal{D} u$ is a nonnegative $L^*$-supersolution in $\mathcal{B}$.
\end{enumerate}
\end{lemma}

It follows from Lemma \ref{lem0800sun}(a) that $\mathbb{P}_\mathcal{D} u$ is an $L^*$-potential if $u$ is.
Moreover, since $u \in \mathscr{F}_{\mathcal{D}}$, Lemma \ref{lem0800sun}(d) implies that $u-\mathbb{P}_\mathcal{D} u$ is a nonnegative $L^*$-supersolution in $\mathcal{B}$.
Combining this with Lemma \ref{lem0800sun}(b), we conclude that $u-\mathbb{P}_\mathcal{D} u$ is a nonnegative $L^*$-solution in $\mathcal{D}$.

The following proposition is the analogue of \cite[Proposition 5.9]{DKK25a}.
The proof is unchanged from that of \cite[Proposition 5.9]{DKK25a}.

\begin{proposition}			\label{prop0800tue}
Let $u$ be a nonnegative $L^*$-supersolution in $\mathcal B$, and assume that $u<+\infty$ in $\mathcal B$.
Then, for each $x\in\mathcal B$, there exists a Radon measure $\mu^x$ on $\mathcal B$ such that
\[
\mu^x(\mathcal{D})=\mathbb{P}_{\mathcal{D}} u(x)
\]
for every open set $\mathcal D\subset\mathcal B$.
\end{proposition}

\begin{theorem}	\label{thm_capacity_measure}
Let $K$ be a compact subset of $\mathcal{B}$.
There exists a Borel measure $\mu=\mu_K$, called the capacitary measure of $K$, supported on $\partial K$, such that
\[
\hat u_K(x)=\int_{K} G^*(x, y)\,d\mu(y), \qquad x \in \mathcal{B},
\]
where $G^*(x,y)$ is the Green function for $L^*$ in $\mathcal{B}$.

We define the $L^*$-capacity of $K$, or simply the capacity of $K$, by
\[
\capacity(K)=\mu(K).
\]
\end{theorem}

\begin{proof}
The proof is essentially the same as that of \cite[Theorem 5.14]{DKK25a}.
By Lemma \ref{lem1013sat}, $\hat u_K$ is an $L^*$-potential in $\mathcal{B}$.
Hence Proposition \ref{prop0800tue} implies that there exists a family of measures $\{\nu^x\}_{x\in \mathcal{B}}$ such that, for every open set $\mathcal{D} \subset \mathcal{B}$,
\[
\mathbb{P}_{\mathcal D}\hat u_K(x)=\nu^x(\mathcal D).
\]

We first show that $\nu^x$ is supported in $\partial K$.
Let $y_0\in \intr (K)$.
Choose $r>0$ such that $B_r(y_0) \Subset \intr(K)$.
By Lemma \ref{lem1013sat}, we have $L^* \hat u_K= L^*W=0$ in $B_r(y_0)$.
Since $v\equiv 0$ satisfies the condition in Lemma \ref{lem0800sun}(d), with $u=\hat{u}_K$, we obtain 
\[
0-\mathbb{P}_{B_r(y_0)}\hat u_K \ge 0.
\]
Because $\mathbb{P}_{B_r(y_0)}\hat u_K \ge 0$, it follows that $\mathbb{P}_{B_r(y_0)}\hat u_K=0$.
Similarly, if $y_0 \in \mathcal{B}\setminus K$, we choose $r>0$ such that  $B_r(y_0) \Subset \mathcal{B} \setminus K$.
By Lemma \ref{lem1013sat}, we have $L \hat u_K=0$ in $B_r(y_0)$, and the same argument yields $\mathbb{P}_{B_r(y_0)}\hat u_K=0$.
Therefore, if $y_0 \notin \partial K$, then there exists a ball $B_r(y_0)$ such that
\[
\nu^x(B_r(y_0))=0.
\]
Hence $\nu^x$ is supported on $\partial K$.

Let $x_0$, $y \in \mathcal{B}$ with $x_0 \neq y$.
Consider the family of balls $B_r(y)$, where $0<r<r_0$ and  $r_0:=\dist(y,\partial \mathcal{B})$.
By Lemma \ref{lem0800sun}, $\mathbb{P}_{B_r(y)}\hat{u}_{K}$ is a nonnegative $L^*$-supersolution in $\mathcal B$ that vanishes continuously on $\partial \mathcal{B}$.
Hence, by the strong minimum principle, either $\mathbb{P}_{B_r(y)}\hat{u}_{K}>0$ in $\mathcal{B}$ or $\mathbb P_{B_r(y)} \hat u_K\equiv 0$.
In the case $\mathbb{P}_{B_r(y)}\hat{u}_{K} \not\equiv 0$, we define
\[
v_r(x):=\frac{\mathbb{P}_{B_r(y)}\hat{u}_{K}(x)}{\mathbb{P}_{B_r(y)}\hat{u}_{K}(x_0)}.
\]

By Lemma \ref{lem40} below, if $\nu^{x_0}(B_r(y))>0$ for all $r \in (0,r_0)$, then
\[
\lim_{r\to 0} v_r(x)=\lim_{r\to 0}\frac{\mathbb{P}_{B_r(y)}\hat{u}_{K}(x)}{\mathbb{P}_{B_r(y)}\hat{u}_{K}(x_0)}=\frac{G^*(x,y)}{G^*(x_0,y)},\qquad  x \neq x_0.
\]
The same conclusion is trivially valid for $x=x_0$.
Hence
\[
\lim_{r\to 0}\frac{\nu^x(B_r(y))/G^*(x,y)}{\nu^{x_0}(B_r(y))/G^*(x_0,y)}=1,
\]
provided that $\nu^{x_0}(B_r(y)) >0$ for all $r \in (0, r_0)$.

Therefore, define a family of measures $\{\mu^x\}_{x\in \mathcal{B}}$ by
\[
d\mu^x(y)=\frac{1}{G^*(x,y)}\,d\nu^x(y).
\]
Then each $\mu^x$ is well-defined and satisfies $\mu^x=\mu^{x_0}$.
In other words, there exists a measure $\mu$, independent of $x$, such that for every $x\in\mathcal B$,
\[
d\mu(y)=\frac{1}{G^*(x,y)}\,d\nu^x(y).
\]
Since $\nu^x(\mathcal B)=\mathbb{P}_{\mathcal{B}}\hat{u}_{K}(x)=\hat{u}_{K}(x)$, we obtain
\[
\hat{u}_K(x)=\nu^x(\mathcal B)=\int_{\mathcal B} 1 \,d\nu^x(y)=\int_{\mathcal B} G^*(x,y)\,d\mu(y).
\]

Finally, since $\nu^x$ is supported on $\partial K$, the same is true of $\mu$.
This completes the proof of the theorem.
\end{proof}

\begin{lemma}			\label{lem40}
Suppose that $\nu^{x_0}(B_r(y))>0$ for all $r \in (0,r_0)$.
Then the family $\{v_r\}_{0<r<r_0}$ converges locally uniformly in $\mathcal{B}\setminus \{x_0\}$ to $G^*(\,\cdot,y)/G^*(x_0,y)$ as $r\to 0$.
\end{lemma}

\begin{proof}
Let $\epsilon \in (0,r_0)$.
For each $r \in (0,\epsilon)$, the function $v_r$ is an $L^*$-solution in $\mathcal{B} \setminus \overline B_\epsilon(y)$.
Since $v_r(x_0)=1$, the Harnack inequality, Theorem \ref{thm02}, implies that the family $\{v_r\}_{0<r<\epsilon}$ is uniformly bounded on every compact subset of $\mathcal{B} \setminus \overline B_\epsilon(y)$.

Hence, by a standard diagonalization argument, there exists a sequence $\{r_k\}_{k=1}^\infty$ with $r_k\to 0$ such that $v_{r_k}$ converges weakly to a function $v$ in $L^p_{\rm loc}(\mathcal{B}\setminus \{y\})$ for every $p \in (1,\infty)$.
It is straightforward to verify that $v$ satisfies
\[
v \in C(\mathcal{B} \setminus \{y\}),\quad L^*v =0\;\text{ in }\;\mathcal{B}\setminus \{y\},\quad v=0\;\text{ on }\;\partial\mathcal{B},\quad v(x_0)=1.
\]

We show that $v=kG^*(\,\cdot, y)$ for some nonnegative constant $k$.
Define $k$ by
\[
k:=\sup\, \{ \alpha \in \mathbb{R}: v(x)-\alpha G^*(x,y) \ge 0\;\text{ for all }\;x\in \mathcal{B},\;x\neq y\}.
\]
Since $v \ge 0$ in  $\mathcal B\setminus \{y\}$, we have $k \ge 0$.
Moreover, $k<+\infty$.
Indeed, if $\alpha$ belongs to the set above, then evaluating at $x_0$ gives
\[
0\le v(x_0)-\alpha G^*(x_0,y)= 1-\alpha G^*(x_0,y) \implies \alpha\le \frac{1}{G^*(x_0,y)}.
\]
Therefore, $k$ is a nonnegative real number.
By the definition of $k$, we have
\begin{equation}	\label{eq0804tue}
v-kG^*(\,\cdot,y) \ge 0\quad \text{in }\; \mathcal{B}\setminus \{y\}.
\end{equation}

Next, we show that, for every $\epsilon>0$, 
\[
v-(k+\epsilon)G^*(\,\cdot,y) \le 0\quad \text{in }\;\mathcal{B}\setminus \{y\}.
\]
Letting $\epsilon\to 0$ and recalling \eqref{eq0804tue}, we conclude that $v=kG(\,\cdot,y)$.

Suppose, to the contrary, that there exists $x_1 \neq y$ such that
\[
v(x_1)-(k+\epsilon)G^*(x_1,y) >0.
\]
The function $v-(k+\epsilon)G^*(\,\cdot,y)$ is an $L^*$-solution in $\mathcal{B}\setminus \{y\}$, and vanishes on $\partial \mathcal{B}$.
Set 
\[
r_1=\min \left\{\abs{x_1-y}, \tfrac12\dist(y,\partial\mathcal{B}), r_0\right\}.
\]
Then, by the comparison principle, for every $r \in (0,r_1)$,
\begin{equation}			\label{eq0903mon}
0<\sup_{\partial B_r(y)} \left(v-(k+\epsilon)G^*(\,\cdot, y)\right) \le \sup_{\partial B_r(y)} \left(v-kG^*(\,\cdot, y)\right) -\inf_{\partial B_r(y)} \epsilon G^*(\,\cdot, y).
\end{equation}

We note that both $v-kG^*(\,\cdot, y)$ and $\epsilon G^*(\,\cdot, y)$ are nonnegative $L^*$-solutions in $\mathcal{B}\setminus \{y\}$.
Moreover, any two points on $\partial B_{r}(y)$ can be connected by a chain of at most $n_0=n_0(d)$ balls with radius $r/2$, all contained in $B_{2r}(y)\setminus B_{r/2}(y)$.
Therefore, applying the Harnack inequality, Theorem \ref{thm02}, successively to $v-kG^*(\,\cdot, y)$, we obtain
\begin{equation}			\label{eq0904mon}
\sup_{\partial B_r(y)} \left(v-kG^*(\,\cdot, y)\right) \le N\inf_{\partial B_r(y)} \left(v-kG^*(\,\cdot, y)\right),
\end{equation}
where $N$ is independent of $r$.
The same inequality holds for $\epsilon G^*(\,\cdot,y)$.
Therefore, from \eqref{eq0903mon} and \eqref{eq0904mon}, we obtain
\[
0<N \inf_{\partial B_r(y)} \left(v-kG^*(\,\cdot, y)\right) - N^{-1} \sup_{\partial B_r(y)}  \epsilon G^*(\,\cdot, y),\quad \forall r \in (0, r_1).
\]
Consequently, for every  $x \in B_{r_1}(y)\setminus \{y\}$, we have
\[
0 \le v(x)-(k+\epsilon/N^2)G^*(x,y).
\]
Since $v-(k+\epsilon/N^2)G^*(\,\cdot,y)=0$ on $\partial\mathcal{B}$, the comparison principle implies that the same inequality holds for all $x\in \mathcal{B}$, ($x \neq y$).
This contradicts the definition of $k$.

We have shown that $v=kG^*(\,\cdot, y)$.
Since $v(x_0)=1$, it follows that $k=1/G^*(x_0,y)$.
Thus $v=G^*(\,\cdot,y)/G^*(x_0,y)$.
The limit is independent of the subsequence $r_k$.
Hence the whole family $\{v_r\}_{0<r<r_0}$ converges to $G^*(\,\cdot,y)/G^*(x_0,y)$.
Moreover, by the local continuity estimates for $L^*$-solutions, the convergence is locally uniform in $\mathcal B\setminus\{y\}$; see \cite{DK17, DEK18}.
This completes the proof.
\end{proof}

Our definition of the capacitary measure is consistent with the equilibrium measure introduced in \cite{GW82} when $L^*=\Delta$, since in that case $W\equiv1$.
For further details, see \cite[Remark 5.18]{DKK25a}.

Theorem \ref{thm_equivalence} below states that, under Conditions \ref{cond1} and \ref{cond2}, the capacity is comparable to the corresponding capacity for the Laplacian.
We need the following lemmas for its proof.

\begin{lemma}		\label{lem2020sat}
Let $K$ be a compact subset of $\mathcal B$, and let $\mathfrak M_K$ denote the set of all Borel measures $\nu$ supported on $K$ such that
\[
\int_K G^*(x,y)\,d\nu(y) \le W(x)
\qquad
\text{for every }\; x\in\mathcal B.
\]
Then
\[
\hat u_{K}(x)=\sup_{\nu \in \mathfrak{M}_K} \int_K G^*(x,y)\,d\nu(y),\qquad x\in\mathcal B.
\]
\end{lemma}

\begin{proof}
Let  $g_\nu(x):=\int_K G^*(x,y)\,d\nu(y)$.
By Lemma \ref{lem1013sat}(a) and Theorem \ref{thm_capacity_measure}, we have
\[
\hat{u}_K(x) \le \sup_{\nu \in \mathfrak{M}_K} g_{\nu}(x)=\sup_{\nu \in \mathfrak{M}_K} \int_K G^*(x,y)\,d\nu(y).
\]

Next, let $v \in \mathfrak{S}^+(\mathcal{B})$ satisfy $v\ge 0$ in $\mathcal{B}$ and $v \ge W$ on $K$.
Then, for any $\nu \in \mathfrak{M}_K$, we have $g_\nu \le W \le v$ in $K$.
In particular, $v \ge W \ge g_\nu$ on $\partial K$.
Since $v \ge 0$ in $\mathcal B$ and $g_\nu$ vanishes continuously on $\partial \mathcal B$, we have $v \ge g_{\nu}$ on $\partial \mathcal{B}$.
Therefore,
\[
v \ge g_{\nu}\quad \text{on }\;\partial(\mathcal{B}\setminus K).
\]
Since $v$ is an $L^*$-supersolution in $\mathcal{B}$ and $g_\nu$ is an $L^*$-solution in $\mathcal{B}\setminus K$, the comparison principle gives
\[
v \ge g_{\nu}\quad \text{in }\;\mathcal{B} \setminus K.
\]

Hence, $v \ge g_{\nu}$ in $\mathcal{B}$.
Recalling Definition \ref{def_rfn} and taking the infimum over all such $v$, we obtain
\[
u_K(x) \ge g_{\nu}(x)=\int_K G^*(x,y)\,d\nu(y),\qquad  x \in \mathcal{B}.
\]

We next pass to the lower semicontinuous regularization.
By the definition of $\hat u_K$, we have
\begin{align*}
\hat{u}_K (x) &= \sup_{r> 0} \left(\inf_{z \in B_r(x) \cap \mathcal{B}}u_K(z) \right) \geq \lim_{r \to 0} \bigg(\inf_{z \in B_r(x)\cap \mathcal{B}}\int_{K}G^*(z,y)\,d\nu(y)\bigg)\\ &\geq \int_K\liminf_{z \to x}G^*(z,y)\,d\nu(y) =\int_K G^*(x,y)\,d\nu(y),
\end{align*}
where we used Fatou's lemma.
Taking the supremum over $\nu\in \mathfrak{M}_K$, we complete the proof.
\end{proof}

The following lemma gives a convenient characterization of regular boundary
points in terms of capacitary potentials.

\begin{lemma}		\label{lem_charact}
Let $\Omega \Subset \mathcal{B}$ and let $x_0 \in \partial \Omega$.
For
\[
0<r<r_0:=\tfrac12\dist(x_0,\partial\mathcal B),
\]
set
\[
B_r=B_r(x_0), \qquad E_r=\overline B_r \setminus\Omega.
\]
Then $x_0$ is regular if and only if
\[
\hat u_{E_r}(x_0)=W(x_0)
\]
for every $r\in(0,r_0)$.
\end{lemma}

\begin{proof}
\noindent
\textbf{(Necessity)}
Suppose that $\hat{u}_{E_r}(x_0)<W(x_0)$ for some $r \in (0,r_0)$.
Define
\[
f(x)=\left(1-\abs{x-x_0}/r\right)_+ W(x),\qquad t_+:=\max(t,0).
\]
Let $u$ be the lower Perron solution to the Dirichlet problem
\[
L^* u=0\;\text{ in }\;\Omega,\qquad  u=f\;\text{ on }\;\partial\Omega.
\]

By Definition \ref{def_rfn} and Lemma \ref{lem1013sat}(d), we have
\[
u \le u_{E_r}=\hat u_{E_r}\;\text{ in }\;\Omega.
\]
This contradicts the regularity at $x_0$, since
\[
\liminf_{x\to x_0,\,x\in \Omega}\ u(x) \le  \liminf_{x\to x_0}\, \hat u_{E_r}(x)= \hat u_{E_r}(x_0) < W(x_0)=f(x_0),
\]
where we used Lemma \ref{lem1013sat}(e) and Lemma \ref{lem2334sun}(c).
 
\medskip
\noindent
\textbf{(Sufficiency)}
Suppose that $\hat{u}_{E_r}(x_0)=W(x_0)$ for every $r\in(0,r_0)$.
Then, by Lemma \ref{lem2334sun}(c), we have 
\begin{equation}			\label{eq1212mon}
\liminf_{x\to x_0} \hat u_{E_r}(x) = W(x_0).
\end{equation}

Let $f \in C(\partial \Omega)$.
Then, for every $\epsilon>0$, there exists $r \in (0,r_0)$ such that
\[
\abs{f(x) - f(x_0)}< \epsilon
\]
for all $x \in \partial \Omega$ satisfying $\abs{x-x_0}<2r$. 
On the other hand, since
\[
L^* \hat{u}_{E_r}=0\quad\text{in }\;\mathcal{B} \setminus \overline B_{3r/2},
\]
and $\hat{u}_{E_r} \le W$ in $\mathcal{B}$, Lemma \ref{lem03} implies that
\[
\frac{\hat u_{E_r}}{W}<1\quad \text{on }\;\partial B_{2r},
\]
because $\hat u_{E_r}=0$ on $\partial\mathcal B$.
Hence
\[
\sup_{\partial B_{2r}} \,\frac{\hat{u}_{E_r}}{W}=1-\delta
\]
for some $\delta \in (0,1)$.
Set
\[
M:=\sup_{\partial \Omega} \,\abs{f-f(x_0)}
\]
and consider
\[
\overline v:=(f(x_0)+\epsilon)W + \frac{M}{\delta}(W- \hat{u}_{E_r}).
\]
Then
\[
\overline v \ge (f(x_0)+\epsilon)W \ge fW\;\text{ on }\;\partial \Omega \cap B_{2r}(x_0).
\]
Moreover, Lemma \ref{lem03} gives
\[
\hat{u}_{E_r} \le (1-\delta)W \quad \text{in }\;\mathcal{B}\setminus B_{2r}(x_0).
\]
Hence
\[
\overline v \ge (f(x_0)+\epsilon)W +MW \ge  fW \quad \text{on }\;\partial\Omega \setminus B_{2r}(x_0).
\]
Therefore, $\overline v \ge fW$ on $\partial \Omega$.

Let $\overline H_{fW}$ be the upper Perron solution, in the sense of Definition \ref{def_perron}, of the problem
\[
L^* u=0\quad\text{in }\;\Omega,\qquad  u=fW\quad\text{on }\;\partial \Omega.
\]
Since $\overline v$ is an $L^*$-solution in $\Omega$ and $\overline v\ge fW$ on $\partial\Omega$, we have $\overline H_{fW} \le \overline v$ in $\Omega$.
Thus,
\begin{align*}
\limsup_{x \to x_0,\, x \in \Omega}\overline H_{fW}(x) &\le \limsup_{x \to x_0,\, x \in \Omega}  \overline v(x) = (f(x_0) + \epsilon)W(x_0) + \frac{M}{\delta}(W(x_0)-\liminf_{x \to x_0} \hat{u}_{E_r}(x))\\
&=(f(x_0) + \epsilon)W(x_0),
\end{align*}
where we used Lemma \ref{lem1013sat}(e) and \eqref{eq1212mon}.
Similarly, define 
\[
\underline v:=(f(x_0)-\epsilon)W + \frac{M}{\delta}(\hat{u}_{E_r}-W).
\]
Since $\underline v\le fW$ on $\partial\Omega$, it follows that $\underline H_{fW} \ge \underline v$ in $\Omega$.
Therefore,
\begin{align*}
\liminf_{x \to x_0,\, x \in \Omega} \underline H_{fW}(x) &\ge \liminf_{x \to x_0,\, x \in \Omega} \underline v(x) = (f(x_0)-\epsilon)W(x_0)+\frac{M}{\delta}(\liminf_{x \to x_0} \hat{u}_{E_r}(x)-W(x_0))\\
&=(f(x_0)-\epsilon)W(x_0).
\end{align*}
Since $\epsilon>0$ is arbitrary, we conclude that
\[
\lim_{x \to x_0,\, x \in \Omega}\underline H_{fW}(x)=\lim_{x \to x_0,\, x \in \Omega}\overline H_{fW}(x)=(fW)(x_0).
\]
Since $f \in C(\partial\Omega)$ is arbitrary and $W$  is continuous and strictly positive on $\partial\Omega$, it follows that  $x_0$ is a regular point.
This completes the proof.
\end{proof}

\section{Wiener criterion: $d \ge 3$}			\label{sec5}
Throughout this section, we assume that $d\ge3$.
Recall that the concentric balls $\mathcal B$, $\mathcal B'$, and $\mathcal B''$ are defined in \eqref{eq_balls}.

\begin{lemma}				\label{lem0956fri}
There exists a constant $N>0$, depending only on $d$, $\lambda$, $\Lambda$, $\omega_{\mathbf A}$, and $R_0$, such that for any compact set $K \subset \mathcal{B}''$ and any point $x^o \in \mathcal{B}\setminus \mathcal{B}'$, we have
\[
N^{-1} \hat u_K(x^o) \le \capacity(K)  \le N\hat u_K(x^o).
\]
\end{lemma}

\begin{proof}
By Theorem \ref{thm_capacity_measure},
\[
\hat u_K(x^o)=\int_{K}G^*(x^o, y)\,d\mu(y),
\]
where $\mu$ is the capacitary measure of $K$.
The two-sided Green function estimates in Theorem \ref{thm_green_function} then give the desired conclusion, since for every $y\in K$,
\[
R_0 \le \abs{x^o-y} \le 3R_0.
\]
This completes the proof.
\end{proof}

\begin{theorem}			\label{thm_equivalence}
Let $K$ be a compact subset of $\mathcal B''$, and let $\capacity^\Delta(K)$ denote the capacity of $K$ associated with the Laplacian.
Then there exists a constant $N>0$, depending only on $d$, $\lambda$, $\Lambda$, $\omega_{\mathbf A}$, $p_0$, and $R_0$, such that
\[
N^{-1} \capacity^\Delta(K) \le \capacity(K) \le  N \capacity^\Delta(K).
\]
\end{theorem}

\begin{proof}
Let $G^\Delta(x,y)$ be the Green function for the Laplacian in $\mathcal B$.
We denote by $\hat u_K^\Delta$ and $\mu^\Delta$, respectively, the capacitary potential and the capacitary measure of $K$ associated with the Laplacian.

By Theorem \ref{thm_green_function}, there is a constant $\alpha>0$ such that
\begin{equation}		\label{greens}
\alpha G^\Delta(x,y) \le G^*(x,y) \le \alpha^{-1} G^\Delta(x,y),\qquad  x,y\in\overline{\mathcal B'},\quad x\ne y.
\end{equation}
Let $\mu$ be a Borel measure supported on $K$ such that
\begin{equation}		\label{eq1844thu}
\int_K G^\Delta(x, y)\, d\mu(y) \leq 1,\qquad x \in \mathcal{B}.
\end{equation}
Set $\tilde\mu:=\gamma_0\alpha \mu$, where $\gamma_0$ is the constant from Lemma \ref{lem02}.
Using \eqref{greens}, we obtain
\begin{equation}			\label{eq1747thu}
\gamma_0^{-1} \int_K G^*(x, y)\,d\tilde\mu(y) \le \int_K G^\Delta(x,y)\, d\mu(y) \le 1,\qquad x \in \overline{\mathcal{B}'}.
\end{equation}
Note that the function
\[
u(x):=\int_K G^*(x, y)\, d\tilde\mu(y)
\]
satisfies $L u=0$ in $\mathcal{B} \setminus K$.
Since $u=0$ on $\partial \mathcal{B}$ and $u \le \gamma_0\le W$ on $\partial K$, the comparison principle gives $u \le W$ in $\mathcal{B} \setminus K$.
Combining this with \eqref{eq1747thu}, we obtain
\[
\int_K G^*(x,y)\, d \tilde\mu(y) \le \gamma_0 \le W,\qquad x \in \mathcal{B}.
\]
Therefore, by Lemma~\ref{lem2020sat}, we obtain
\[
\int_K G^*(x,y) \,d \tilde\mu(y) \le \hat u_K(x),\qquad x \in \mathcal{B}.
\]
Since the measure $\mu$ is supported on $K \subset \mathcal{B}'$, it follows from \eqref{greens} that, for $x \in \overline{\mathcal{B}'}$,
\[
\gamma_0 \int_{K} G^\Delta(x, y)\,d\mu(y) \le \gamma_0\alpha^{-1} \int_K G^*(x,y)\, d\mu(y) = \alpha^{-2} \int_K G^*(x ,y)\,d\tilde{\mu}(y) \le \alpha^{-2}  \hat u_K(x).
\]
Taking the supremum over all measures $\mu$ satisfying \eqref{eq1844thu} and using Lemma~\ref{lem2020sat} again, we obtain
\[
\gamma_0 \hat u^\Delta_K(x) \le \alpha^{-2} \hat u_K(x),\qquad x \in \overline{\mathcal{B}'}.
\]
Choosing $x^o \in \partial\mathcal{B}'$ and applying Lemma~\ref{lem0956fri}, we conclude that
$\capacity^\Delta(K) \le N \capacity(K)$ for some constant $N>0$.
Interchanging the roles of $G^\Delta(x,y)$ and $G^*(x,y)$, we also obtain $\capacity(K) \le N \capacity^\Delta(K)$.
\end{proof}

\begin{lemma}		\label{lem1650thu}
Let $N_0$ be the constant in Theorem \ref{thm_green_function}.
If $K\subset \overline B_r(x_0)$ and $B_{2r}(x_0)\subset\mathcal B'$, then
\[
\capacity(K) \le 3^{d-2} \gamma_0^{-1} N_0 r^{d-2}.
\]
\end{lemma}
\begin{proof}
Observe that
\[
\sup\left\{\,\abs{x-y}: x \in \partial B_{2r}(x_0), y \in K  \right\} \le 3r.
\]
Therefore, by Theorem \ref{thm_green_function}, Lemma \ref{lem1013sat}, Theorem \ref{thm_capacity_measure}, and Lemma \ref{lem02}, for every $x \in \partial B_{2r}(x_0)$ we have
\[
N_0^{-1}(3r)^{2-d} \mu(K)\le \int_K G^*(x,y)\,d\mu(y) =\hat u_K(x) \le W(x) \le \gamma_0^{-1}.
\]
This completes the proof.
\end{proof}

\begin{theorem}[Wiener criterion]		\label{thm_wiener}
Assume that Conditions \ref{cond1} and \ref{cond2} hold.
Let $\Omega \subset \mathcal{B}''$ be an open set.
Then a point $x_0 \in \partial \Omega$ is regular if and only if
\[
\sum_{k=0}^\infty \frac{\capacity(\overline{B_{2^{-k}r_0}(x_0)}\setminus \Omega)}{(2^{-k}r_0)^{d-2}}=+\infty.
\]
\end{theorem}

\begin{proof}
We use the notation
\[
B_r=B_r(x_0),\quad
E_r=\overline B_r\setminus \Omega, \quad  u_r= u_{E_r},\quad \mu_r=\mu_{E_r}.
\]
We also write $G^*(x,y)$ for the Green function of $L^*$ in $\mathcal{B}$.

\medskip
\noindent
\textbf{(Necessity)}
We show that $x_0 \in \partial \Omega$ is not regular if
\begin{equation}			\label{eq0801thu}
\sum_{k=0}^\infty \frac{\capacity(E_{2^{-k}r_0})}{(2^{-k}r_0)^{d-2}}<+\infty.
\end{equation}
By Theorem \ref{thm_capacity_measure}, we have
\begin{equation}			\label{eq1353thu}
\hat u_r(x)=\int_{E_r} G^*(x, y)\,d\mu_r(y).
\end{equation}
Note that for $0<\rho<r$,
\[
\int_{E_\rho}G^*(x_0, y)\,d\mu_r(y)  \le \int_{E_r}G^*(x_0, y)\,d\mu_r(y)=\hat u_r(x_0) \le W(x_0) \le \gamma_0^{-1}.
\]
Since the sets $E_\rho$ decrease to $\{x_0\}$ as $\rho\to 0$, it follows that
\[
\lim_{\rho \to 0} \int_{E_\rho}G^*(x_0, y)\,d\mu_r(y) = \int_{\{x_0\}} G^*(x_0,y)\,d\mu_r(y) \le \gamma_0^{-1}.
\]
Because $G^*(x_0,x_0)=+\infty$, we must have $\mu_r(\{x_0\})=0$, and hence
\begin{equation}			\label{eq1445thu}
\int_{\{x_0\}} G^*(x_0,y)\,d\mu_r(y)=0.
\end{equation}
Therefore, partitioning $E_r$ into annular regions and using the Green function estimates in Theorem~\ref{thm_green_function}, we obtain from \eqref{eq1353thu} and \eqref{eq1445thu} that
\begin{align}
			\nonumber
\hat u_r(x_0) &= \sum_{k=0}^\infty \int_{E_{2^{-k}r}\setminus E_{2^{-k-1}r}} G^*(x_0, y)\,d\mu_r(y)+ \int_{\{x_0\}} G^*(x_0,y)\,d\mu_r(y)\\
			\label{eq1725sat} 
& \lesssim \sum_{k=0}^\infty(2^{-k}r)^{2-d} \mu_r(E_{2^{-k}r})
\end{align}

We claim that there exists a constant $C$ such that
\begin{equation}			\label{eq1727sat}
\mu_r(E_t) \le C \mu_t(E_t)=C \capacity(E_t), \quad \forall t \in (0,r].
\end{equation}
Assuming \eqref{eq1727sat} and taking $r=2^{-n} r_0$ in \eqref{eq1725sat}, we obtain
\[
\hat u_{2^{-n} r_0}(x_0) \lesssim \sum_{k=0}^\infty (2^{-k-n}r_0)^{2-d} \capacity(E_{2^{-k-n} r_0}) = \sum_{k=n}^\infty (2^{-k}r_0)^{2-d} \capacity(E_{2^{-k} r_0}).
\]
By assumption \eqref{eq0801thu}, the last expression tends to zero as $n \to \infty$.
Hence we may choose $n$ sufficiently large so that
\[
\hat u_{2^{-n} r_0}(x_0)<W(x_0).
\]
Lemma~\ref{lem_charact} then implies that $x_0$ is not regular.

It remains to prove \eqref{eq1727sat}.
Fix $0<t \le r$, and define a measure $\nu$ by
\[
\nu(E)=\mu_r(E_t \cap E).
\]
Then $\nu$ is supported in $E_t$.
Moreover, for every $x \in \mathcal{B}$,
\[
\int_{E_t} G^*(x,y)\,d\nu(y) =\int_{E_t} G^*(x,y)\,d\mu_r(y) \le \int_{E_r} G^*(x,y)\,d\mu_r(y)=\hat u_{E_r}(x) \le W(x).
\]
Therefore, by Lemma \ref{lem2020sat},
\[
\int_{E_t} G^*(x,y)\,d\mu_r(y) \le \hat u_{E_t}(x)=\int_{E_t} G^*(x,y)\,d\mu_t(y),\quad \forall x \in \mathcal{B}.
\]
Taking $x=x^o \in \partial \mathcal{B}'$ in the preceding inequality and using Theorem~\ref{thm_green_function}, we obtain \eqref{eq1727sat}.
This completes the proof of necessity.

\medskip
\noindent
\textbf{(Sufficiency)}
Let $\rho \in(0,r_0)$ be fixed but arbitrarily small.
The following lemma is the analogue of \cite[Lemma 6.8]{DKK25a}.

\begin{lemma}			\label{lem2119thu}
There exist constants $\alpha_0 \in (0,\frac12)$ and $\kappa>0$ such that, for every $\alpha \in (0,\alpha_0)$ and every $r\in (0, \alpha \rho]$,
\[
\sup_{\Omega\cap B_r(x_0)}
\left(1-\frac{u_{E_\rho}}{W}\right)
\le
\exp\left\{
-\kappa
\sum_{i=1}^{j}
(\alpha^i\rho)^{2-d}\capacity(E_{\alpha^i\rho})
\right\},\qquad j:=\left\lfloor \frac{\log(r/\rho)}{\log\alpha}\right\rfloor.
\]

\end{lemma}
\begin{proof}
Throughout the proof, we write
\[
v:=W -u_{E_\rho}.
\]
By Definition \ref{def_rfn}, $v$ admits the characterization
\begin{equation}			\label{eq0704sun}
v(x)=\sup\left\{w(x):w \in \mathfrak{S}^-(\mathcal{B}),\; w \le W\;\text{ in }\;\mathcal{B},\;\; w \le 0\; \text{ on }\; E_\rho\right\}.
\end{equation}
By Lemma \ref{lem1013sat}, we have
\[
0\le v \le W.
\]

Let $0<\alpha<\frac12$ and  $0<t \le \rho$.
By Theorem \ref{thm_green_function}, Theorem \ref{thm_capacity_measure}, and Lemma \ref{lem02}, for each $x \in \partial B_t(x_0)$, we have
\begin{align*}
W(x)- \hat u_{E_{\alpha t}}(x) &=W(x)-\int_{E_{\alpha t}} G^*(x,y)\,d\mu_{\alpha t}(y) \\
& \ge W(x)-N_0 ((1-\alpha)t)^{2-d} \capacity(E_{\alpha t}) \\
& \ge W(x)\left(1-N_0\gamma_0^{-1} 2^{d-2} t^{2-d} \capacity(E_{\alpha t})\right).
\end{align*}

Consequently, for $x\in \Omega \cap \partial B_{t}(x_0)$,
\begin{align}
		\nonumber
(W-\hat u_{E_{\alpha t}})(x) \cdot \sup_{\Omega \cap B_t(x_0)} (v/W) &\ge W(x)\left\{1- N_0\gamma_0^{-1}2^{d-2} t^{2-d} \capacity(E_{\alpha t})\right\} (v/W)(x),\\
		\label{eq2108fri}
& = \left\{1- N_0\gamma_0^{-1}2^{d-2} t^{2-d} \capacity(E_{\alpha t})\right\} v(x).
\end{align}
Since \eqref{eq0704sun} implies that $v = 0$ on $E_\rho$, the inequality \eqref{eq2108fri} also holds for $x \in \partial\Omega \cap \overline B_{t}(x_0)$.
By the comparison principle, we conclude that \eqref{eq2108fri} holds for all $x \in \Omega \cap B_t(x_0)$.

On the other hand, by Theorems \ref{thm_green_function} and \ref{thm_capacity_measure}, for $x \in \partial B_{2\alpha t}(x_0)$ we have
\begin{equation}			\label{eq2109fri}
\hat u_{E_{\alpha t}}(x) \ge N_0^{-1} (3\alpha t)^{2-d} \capacity(E_{\alpha t}).
\end{equation}

Applying Lemma \ref{lem1650thu}, we have
\[
1-N_0\gamma_0^{-1}2^{d-2} t^{2-d} \capacity(E_{\alpha t}) \ge 1-6^{d-2}N_0^2 \gamma_0^{-2} \alpha^{d-2}>0
\]
provided
\[\alpha<\tfrac16 (\gamma_0^2 N_0^{-2})^{1/(d-2)}.
\]
Thus, for $x \in \Omega\cap \partial B_{2\alpha t}(x_0)$, the inequalities \eqref{eq2108fri} and \eqref{eq2109fri} imply
\begin{align*}			
v(x) &\le W(x)\sup_{\Omega \cap B_t(x_0)} \frac{v}{W} \;\;\left\{ \frac{1- N_0^{-1} \gamma_0 3^{2-d}(\alpha t)^{2-d} \capacity(E_{\alpha t})}{1- N_0\gamma_0^{-1}2^{d-2} t^{2-d} \capacity(E_{\alpha t})}\right\}\\
&\le W(x)\sup_{\Omega \cap B_t(x_0)} \frac{v}{W} \;\; \left\{1- \frac{N^{-1}\gamma_0 3^{2-d}}{2}(\alpha t)^{2-d} \capacity(E_{\alpha t})\right\},
\end{align*}
where the last inequality holds whenever
\[
0<\alpha \le \tfrac16 (2N_0^2\gamma_0^{-2})^{-1/(d-2)}.
\]
Again, by \eqref{eq0704sun} and the comparison principle, the preceding inequality remains valid for every $x \in \Omega \cap B_{2\alpha t}(x_0)$.
With the notation
\[
M(r):=\sup_{\Omega \cap B_r(x_0)} \frac{v}{W},
\]
we have shown that there exist constants $\alpha_0 \in (0,\frac12)$ and $\kappa>0$ such that, for $t \in (0, \rho]$ and  $\alpha \in (0, \alpha_0)$,
\[
M(\alpha t) \le M(t) \left\{1-\kappa(\alpha t)^{2-d} \capacity(E_{\alpha t}) \right\}.
\]
Using the inequality $\ln (1-s) \le -s$ for $s <1$, we obtain
\[
M(\alpha t) \le M(t) \exp\left\{-\kappa (\alpha t)^{2-d} \capacity(E_{\alpha t}) \right\}.
\]
Iterating this estimate gives
\[
M(\alpha^j \rho) \le M(\rho) \exp\left\{-\kappa \sum_{i=1}^j (\alpha^i \rho)^{2-d} \capacity(E_{\alpha^i \rho}) \right\},\quad j=1,2,\ldots.
\]

For $r \in (0,\alpha \rho]$, choose $j \ge 1$ such that $r \in (\alpha^{j+1} \rho, \alpha^j \rho]$.
Then
\[
M(r) \le M(\alpha^j \rho) \le M(\rho) \exp\left\{-\kappa \sum_{i=1}^j (\alpha^i \rho)^{2-d} \capacity(E_{\alpha^i \rho}) \right\}.
\]
Since
\[
j=\left\lfloor \frac{\log(r/\rho)}{\log\alpha}\right\rfloor,
\]
the proof is complete.
\end{proof}

By Lemma~\ref{lem2119thu}, for every $\alpha\in(0,\alpha_0)$ we have
\[
W(x)-u_\rho(x) \le
W(x)\exp\left\{-\kappa \sum_{i=1}^{j} (\alpha^i \rho)^{2-d} \capacity(E_{\alpha^i \rho}) \right\},\qquad j=\left\lfloor \frac{\log(\abs{x-x_0}/\rho)}{\log\alpha} \right\rfloor,
\]
for all $x\in\Omega\cap B_{\alpha\rho}(x_0)$.
Thus, for all sufficiently small $r>0$, we obtain
\[
\limsup_{x \rightarrow x_0} \left(W(x)-\hat u_r(x)\right)=0.
\] 
Since $W-\hat u_r \ge 0$ and $W$ is continuous in $\mathcal{B}$, it follows that
\[
\liminf_{x \rightarrow x_0} \hat u_r(x)=W(x_0).
\]
Consequently, by Lemma \ref{lem2334sun}(c), we have
\[
\hat u_{E_r}(x_0)= W(x_0).
\]
Therefore, Lemma \ref{lem_charact} implies that $x_0$ is regular.
This completes the proof of sufficiency, and hence the proof of the theorem.
\end{proof}

The following theorem is an immediate consequence of Theorems \ref{thm_wiener} and \ref{thm_equivalence}.

\begin{theorem}		\label{thm0800sat}
Let $\Omega$ be a bounded domain in $\mathbb R^d$, $d \ge 3$.
Assume that Conditions \ref{cond1} and \ref{cond2} hold.
Then a point $x_0\in\partial\Omega$ is regular with respect to $L^*$ if and only if it is regular with respect to the Laplacian.
\end{theorem}

\begin{definition}\label{def_reg}
We define a domain $\Omega$ as a regular domain if every point on its boundary $\partial \Omega$ is a regular point with respect to the Laplacian.
\end{definition}

\begin{theorem}		\label{thm0802sat}
Let $\Omega$ be a bounded regular domain in $\mathbb R^d$, $d \ge 3$.
Assume that Conditions \ref{cond1} and \ref{cond2} hold.
Then, for every $f\in C(\partial\Omega)$, the Dirichlet problem
\[
L^*u=0 \;\text{ in }\;\Omega, \qquad u=f \;\text{ on }\;\partial\Omega
\]
has a unique solution $u\in C(\overline\Omega)$.
\end{theorem}

\section{Wiener criterion: $d = 2$}			\label{sec_2d}

In the two-dimensional setting, it is convenient to consider the capacitary potential of $E$ relative to a domain $\mathcal D\subset\mathcal B$, which is not fixed in advance.
In our application, $\mathcal D$ will also be a ball. We denote by $G_{\mathcal D}^*(x,y)$ the Green function for the operator $L^*$ in $\mathcal D$.

\begin{definition}			\label{def_rfn2d}
For $E \Subset \mathcal{D}$, define
\[
u_{E,  \mathcal{D}}(x):=\inf \,\left\{v(x):  v \in \mathfrak{S}^+(\mathcal{D}),\; v\ge 0\text{ in }\mathcal{D},\; v \ge W \text{ in }E\right\},\quad x \in \mathcal{D}.
\]
The lower semicontinuous regularization $\hat u_{E,\mathcal{D}}$ is called the capacitary potential of $E$ relative to $\mathcal{D}$.
\end{definition}

\begin{lemma}				\label{lem1112mon} 
The following properties hold:
\begin{enumerate}[leftmargin=*, label=(\alph*)]
\item		\label{item_a}
$0\le \hat u_{E,\mathcal{D}} \le u_{E,\mathcal{D}} \le W$.
\item
$u_{E,\mathcal{D}}=W$ on $E$ and $\hat u_{E,\mathcal{D}} =W$ in $\intr (E)$.
\item
$\hat u_{E,\mathcal{D}}$ is a potential in $\mathcal{D}$;
that is, it is a nonnegative $L^*$-supersolution in $\mathcal{D}$, finite at every point of $\mathcal{D}$, and vanishes continuously on $\partial \mathcal{D}$.
\item
$u_{E,\mathcal{D}}=\hat u_{E,\mathcal{D}}$ in $\mathcal{D}\setminus \overline{E}$, and
$L^* u_{E,\mathcal{D}}=L^* \hat u_{E,\mathcal{D}}=0$ in $\mathcal{D}\setminus \overline{E}$.
\item		\label{item_e}
If $x_0 \in \partial E$, then
\[
\liminf_{x\to x_0} \hat u_{E,\mathcal{D}}(x) =\liminf_{x\to x_0,\, x\in \mathcal{D} \setminus  E} \hat u_{E,\mathcal{D}}(x).
\]
\item		\label{item_f}
If $E_1 \subset E_2 \subset \mathcal{D}$, then
\[\hat u_{E_1,\mathcal{D}} \le \hat u_{E_2,\mathcal{D}}.\]
\item		\label{item_g}
If $E\subset\mathcal D_1\subset\mathcal D_2$, then
\[\hat u_{E,\mathcal{D}_1} \le \hat u_{E,\mathcal{D}_2}.\]
\end{enumerate}
\end{lemma} 

\begin{proof}
The proof of \ref{item_a} -- \ref{item_e} is the same as that of Lemma \ref{lem1013sat}.
Properties \ref{item_f} and \ref{item_g} follow directly from Definition \ref{def_rfn2d} and the comparison principle.
\end{proof}

The following lemma is a variant of Lemma~\ref{lem_charact}.

\begin{lemma}		\label{lem_charact2d}
Let $\Omega \Subset \mathcal{B}$ and let $x_0 \in \partial \Omega$.
Denote
\[
B_r=B_r(x_0), \quad E_r = \overline B_r \setminus\Omega.
\]
Then $x_0$ is regular if and only if
\[
\hat{u}_{E_r,B_{4r}}(x_0)=W(x_0)
\]
for every $r$ satisfying $0<r<r_0:=\tfrac{1}{4}\dist(x_0,\partial\mathcal{B})$.
\end{lemma}

\begin{proof}
\noindent
\textbf{(Necessity)}
Suppose that $\hat{u}_{E_r,B_{4r}}(x_0)<W(x_0)$ for some $r \in (0,r_0)$.
Define
\[
f(x):=\left(1-\abs{x-x_0}/r\right)_+W(x),\qquad t_+:=\max(t,0).
\]
Let $u$ be the lower Perron solution of the Dirichlet problem
\[
L^*u=0\quad\text{in }\;\Omega\cap B_{4r},\qquad  u=f\quad\text{on }\;\partial(\Omega\cap B_{4r}).
\]
By Definition \ref{def_rfn2d} and Lemma \ref{lem1112mon}(d), we have
\[
u \le u_{E_r, B_{4r}}=\hat u_{E_r, B_{4r}} \quad\text{in }\;\Omega\cap B_{4r}.
\]
Hence
\[
\liminf_{x\to x_0,\,x\in \Omega\cap B_{4r}} u(x) \le \liminf_{x\to x_0}\hat u_{E_r,B_{4r}}(x) =\hat u_{E_r,B_{4r}}(x_0) <W(x_0)=f(x_0),
\]
where we used Lemma~\ref{lem1112mon}(e) and Lemma~\ref{lem2334sun}(c).
Thus $x_0$ is not regular with respect to $\Omega\cap B_{4r}$.

Suppose, to the contrary, that $x_0$ is regular with respect to $\Omega$.
Choose $\rho_0>0$ sufficiently small so that $\rho_0<4r$ and Lemma~\ref{lem_charact} applies for $0<\rho<\rho_0$.
Then
\[
\hat u_{E_\rho}(x_0)=W(x_0), \qquad 0<\rho<\rho_0.
\]
Since
\[
\overline B_\rho \setminus(\Omega\cap B_{4r}) = \overline B_\rho \setminus\Omega =E_\rho \quad\text{for }\;0<\rho<\rho_0,
\]
Lemma~\ref{lem_charact}, applied to $\Omega\cap B_{4r}$, shows that $x_0$ is regular with respect to $\Omega\cap B_{4r}$, a contradiction.
Hence $x_0$ is not regular with respect to $\Omega$.

\medskip
\noindent
\textbf{(Sufficiency)}
Suppose that $\hat{u}_{E_r,B_{4r}}(x_0)=W(x_0)$ for every $r\in(0,r_0)$.
By Lemma \ref{lem1112mon}(g), we have
\[
W(x_0)= \hat{u}_{E_r,B_{4r}}(x_0) \leq \hat{u}_{E_r,\mathcal{B}}(x_0) \leq W(x_0).
\]
Hence $\hat{u}_{E_r,\mathcal{B}}(x_0)=W(x_0)$ for every $r \in (0,r_0)$.
Finally, it suffices to apply the sufficient part of Lemma~\ref{lem_charact}.
\end{proof}

\begin{theorem}			\label{thm1639mon}
Let $K$ be a compact subset of $\mathcal{D}$.
Then there exists a Borel measure $\mu_{K, \mathcal{D}}$, called the capacitary measure of $K$ relative to $\mathcal{D}$, supported in $\partial K$, such that
\[
\hat u_{K,\mathcal{D}}(x)=\int_{K} G_{\mathcal{D}}^*(x, y)\,d\mu_{K, \mathcal{D}} (y), \qquad x \in \mathcal{D}.
\]
\end{theorem}

\begin{proof}
The proof is the same as that of Theorem \ref{thm_capacity_measure}.
\end{proof}

\begin{definition}
Let $K$ be a compact subset of $\mathcal{D}$.
The capacity of $K$ relative to $\mathcal{D}$ is defined by
\[
\capacity(K,\mathcal{D}):=\mu_{K,\mathcal{D}}(K).
\]
\end{definition}

\begin{lemma}		\label{lem1125mon}
Let $K$ be a compact subset of $\mathcal{D}$, and let $\mathfrak{M}_{K,\mathcal{D}}$ be the set of all Borel measures $\nu$ supported on $K$ such that
\[
\int_K G_{\mathcal{D}}^* (x,y)\,d\nu(y) \le W(x),\qquad x\in\mathcal D.
\]
Then
\[
\hat u_{K,\mathcal{D}}(x)=\sup_{\nu \in \mathfrak{M}_{K,\mathcal{D}}} \int_K G_{\mathcal{D}}^*(x,y)\,d\nu(y),\qquad x\in\mathcal D.
\]
\end{lemma}
\begin{proof}
The proof is the same as that of Lemma \ref{lem2020sat}.
\end{proof}

\begin{theorem}			\label{thm1641mon}
Let $B_{4r}=B_{4r}(x_0) \subset \mathcal{B}$, and let $K \subset \overline B_r$ be compact.
Denote by $\capacity^{\Delta}(K,B_{4r})$ the capacity of $K$ relative to $B_{4r}$ associated with the Laplacian.
Then there exists a constant $C>0$, depending only on $\lambda$, $\Lambda$, $\omega_{\mathbf A}$, and $R_0$, such that
\[
C^{-1} \capacity^{\Delta}(K, B_{4r}) \le \capacity(K, B_{4r}) \le C \capacity^{\Delta}(K, B_{4r}).
\]
\end{theorem}
\begin{proof}
The proof follows the same argument as that of Theorem \ref{thm_equivalence}, using the Green function estimates from Theorem \ref{thm_green_function2d} in place of those used there.
\end{proof}

The following theorem establishes the Wiener criterion for regular boundary points in two dimensions.

\begin{theorem}[Wiener criterion]		\label{thm_wiener2d}
Assume that Conditions \ref{cond1} and \ref{cond2} hold.
Let $\Omega \Subset \mathcal{B}$ be a domain, and let $x_0 \in \partial \Omega$.
Set $r_0=\frac{1}{8}\dist(x_0, \partial \mathcal{B})$.
Then $x_0$ is regular if and only if
\[
\sum_{k=0}^\infty \capacity \left(\,\overline{B_{2^{-k}r_0}(x_0)}\setminus \Omega,B_{2^{2-k}r_0}(x_0)\right)=+\infty.
\]
\end{theorem}

\begin{proof}
We use the notation
\[
B_r=:B_r(x_0),\qquad E_r:=\overline B_r\setminus \Omega,
\qquad  u_r:= u_{E_r, B_{4r}},\qquad \mu_r:=\mu_{E_r, B_{4r}}.
\]
We also write $G_r^*(x,y)$ for the Green function of $L^*$ in $B_r$.

\medskip
\noindent
\textbf{(Necessity)}
We show that $x_0$ is not regular if
\begin{equation}			\label{eq1305mon}
\sum_{k=0}^\infty \capacity \left(E_{2^{-k}r_0}, B_{2^{2-k}r_0}\right)<\infty.
\end{equation}
By Theorem \ref{thm1639mon}, we have
\begin{equation}			\label{eq1306mon}
\hat u_r(x)=\int_{E_r} G_{4r}^*(x, y)\,d\mu_r(y).
\end{equation}
For $0<\rho<r$, we have
\[
\int_{E_\rho}G_{4r}^*(x_0, y)\,d\mu_r(y)  \le \int_{E_r}G_{4r}^*(x_0, y)\,d\mu_r(y)=\hat u_{r}(x_0) \le W(x_0).
\]
Since the sets $E_\rho$ decrease to $\{x_0\}$ as $\rho \to 0$, it follows that
\[
\lim_{\rho \to 0} \int_{E_\rho}G_{4r}^*(x_0, y)\,d\mu_r(y) = \int_{\{x_0\}} G_{4r}^*(x_0,y)\,d\mu_r(y) \le W(x_0).
\]
Since $G_{4r}^*(x_0,x_0)=\infty$, we must have $\mu_r(\{x_0\})=0$.
Consequently,
\begin{equation}			\label{eq1307mon}
\int_{\{x_0\}} G_{4r}^*(x_0,y)\,d\mu_r(y)=0.
\end{equation}
We now partition $E_r$ into annular regions and use the Green function estimates from Theorem~\ref{thm_green_function2d}.
From \eqref{eq1306mon} and \eqref{eq1307mon}, we obtain
\begin{align*}
\hat u_{r}(x_0) &= \sum_{k=0}^{\infty} \int_{E_{2^{-k}r}\setminus E_{2^{-k-1}r}} G_{4r}^*(x_0, y)\,d\mu_r(y)+ \int_{\{x_0\}} G_{4r}^*(x_0,y)\,d\mu_r(y) \\
& \lesssim \sum_{k=0}^\infty \log \left(\frac{3r}{2^{-k-1}r}\right) \mu_r(E_{2^{-k}r}\setminus E_{2^{-k-1}r}) \simeq \sum_{k=0}^\infty (k+1) \mu_r(E_{2^{-k}r}\setminus E_{2^{-k-1}r}).
\end{align*}
By an elementary summation by parts, we have
\[
\sum_{k=0}^N (k+1)\mu_r(E_{2^{-k}r}\setminus E_{2^{-k-1}r})= \sum_{k=0}^N \mu_r(E_{2^{-k}r}\setminus E_{2^{-N-1}r}).
\]
Consequently,
\[
\hat u_r(x_0) \lesssim \lim_{N\to\infty}  \sum_{k=0}^N (k+1)\mu_r(E_{2^{-k}r}\setminus E_{2^{-k-1}r}) \lesssim \lim_{N\to\infty} \sum_{k=0}^N \mu_r(E_{2^{-k}r} \setminus E_{2^{-N-1}r}).
\]
Since $\mu_r$ is a nonnegative measure, we obtain
\begin{equation}			\label{eq1329mon}
\hat u_r(x_0) \lesssim \lim_{N\to\infty} \sum_{k=0}^N \mu_r(E_{2^{-k}r}) = \sum_{k=0}^{\infty} \mu_r(E_{2^{-k}r}).
\end{equation}

We claim that there exists a constant $C>0$ such that
\begin{equation}			\label{eq1308mon}
\mu_r(E_t) \le C \mu_t(E_t)=C \capacity(E_t,B_{4t}), \quad \forall t \in (0,r].
\end{equation}
Assuming \eqref{eq1308mon} and taking $r=2^{-n}r_0$ in \eqref{eq1329mon}, we obtain
\[
\hat u_{2^{-n} r_0}(x_0) \lesssim \sum_{k=0}^\infty \capacity(E_{2^{-k-n} r_0},B_{2^{2-k-n} r_0}) = \sum_{k=n}^\infty \capacity(E_{2^{-k} r_0},B_{2^{2-k} r_0}).
\]
By hypothesis \eqref{eq1305mon}, the last expression tends to zero as $n\to\infty$.
Hence we may choose $n$ sufficiently large so that
\[
\hat u_{2^{-n} r_0}(x_0)<W(x_0).
\]
Lemma~\ref{lem_charact2d} implies that $x_0$ is not regular.

It remains to prove \eqref{eq1308mon}.
Fix $0<t \le r$, and define a measure $\nu$ by
\[
\nu(E)=\mu_r(E_t \cap E).
\]
Then $\nu$ is supported on $E_t$.
Moreover, for every $x\in B_{4t}$, we have
\begin{align*} 
\int_{E_t} G_{4t}^*(x,y)\,d\nu(y) &=\int_{E_t} G_{4t}^*(x,y)\,d\mu_r(y) \le \int_{E_t} G_{4r}^*(x,y)\,d\mu_r(y)\\ &\le \int_{E_r} G_{4r}^*(x,y)\,d\mu_r(y) \le W(x),
\end{align*}
where we used the domain monotonicity of the Green function,
\[
G_{4t}^*(x,y) \le G_{4r}^*(x,y)  \quad\text{for all }\;x, y \in B_{4t}.
\]

Applying Lemma~\ref{lem1125mon}, we obtain
\[
\int_{E_t} G_{4t}^*(x,y)\,d\mu_r(y) \le \hat u_{t}(x)=\int_{E_t} G_{4t}^*(x,y)\,d\mu_t(y),\quad \forall x \in B_{4t}.
\]
Taking $x \in \partial B_{3t/2}$ in the preceding inequality and using the following variant of \eqref{estimate_green} (see \cite[Remark 3.20]{DKK25b}),
\[
C_1 \log \left(\frac{5t}{2\abs{x-y}}\right) \le G_{4t}^*(x,y)\le C_2 \log \left(\frac{4t}{\abs{x-y}}\right),\qquad x,y\in\overline B_{3t/2},\quad x\ne y,
\]
we obtain \eqref{eq1308mon}.
This completes the proof of the necessity.

\medskip
\noindent
\textbf{(Sufficiency)}
We first state the following lemma, which is similar to Lemma \ref{lem2119thu}, but adapted to the two-dimensional setting.

\begin{lemma}			\label{lem1415mon}
There exists a constant $\gamma>0$ such that, for every $t \in (0, r)$,
\[
\sup_{\Omega\cap B_t(x_0)}\left(1- \frac{\hat u_r}{W} \right)  \le \exp\left\{-\gamma\sum_{k=0}^{\lfloor \log_2(r/t) \rfloor}\capacity(E_{2^{-k}r},B_{2^{2-k}r})\right\}.
\]
\end{lemma}

\begin{proof}
To study the local behavior of $\hat u_r$, we construct a sequence of auxiliary functions as follows.
Define
\[
v_0:=\frac{u_r}{W}, \qquad \hat v_0:=\frac{\hat u_r}{W}, \qquad m_0:=\inf_{B_r}\hat v_0.
\]
For $i=1,2,\ldots$, define
\[
v_i := m_{i-1} + (1-m_{i-1})\, \frac{u_{2^{-i}r}}{W},
\]
and set
\[
m_i := \inf_{B_{2^{-i}r}} \,\hat v_i=m_{i-1} + (1-m_{i-1})\, \inf_{B_{2^{-i}r}}\, \frac{\hat u_{2^{-i}r}}{W},
\]
where $\hat v_i$ denotes the lower semicontinuous regularization of $v_i$.
Using Theorems \ref{thm1639mon} and \ref{thm_green_function2d}, together with Lemma \ref{lem02}, we obtain, for all $x \in B_{2^{-i}r}$,
\begin{equation}			\label{eq0737wed}
\frac{\hat u_{2^{-i}r}(x)}{W(x)}  \ge \gamma_0 \left(\inf_{y \in B_{2^{-i}r}}G_{2^{2-i}r}^*(x, y)\right)\, \capacity(E_{2^{-i}r},B_{2^{2-i}r})  \ge \gamma \capacity(E_{2^{-i}r},B_{2^{2-i}r}),
\end{equation}
where $\gamma:=\gamma_0 C_0^{-1} \log(3/2)>0$.
Moreover,
\[
1-\hat v_i = (1-m_{i-1}) \left(1-\frac{\hat u_{2^{-i}r}}{W}\right).
\]
Taking the supremum over $B_{2^{-i}r}$ in the preceding identity and using \eqref{eq0737wed} together with
\[
1-s\le e^{-s},
\]
we obtain
\[
1-m_i \le (1-m_{i-1}) \left(1-\gamma\capacity(E_{2^{-i}r},B_{2^{2-i}r})\right)
\le (1-m_{i-1}) \exp\left\{-\gamma\capacity(E_{2^{-i}r},B_{2^{2-i}r})\right\}.
\]
Iterating this inequality gives
\begin{equation}		\label{eq1310mon}
\sup_{B_{2^{-i}r}}\, (1-\hat v_i) = 1-m_i
\le \exp\left\{ -\gamma \sum_{k=0}^{i} \capacity(E_{2^{-k}r},B_{2^{2-k}r}) \right\}.
\end{equation}
We claim that
\begin{equation}	\label{eq1311mon}
\hat v_j  \le \frac{\hat u_r}{W} \quad \text{in }\; B_{2^{2-j}r},\quad j=1,2,\ldots .
\end{equation}
Assuming \eqref{eq1311mon} for the moment, let
\[
n=\lfloor \log_2(r/t) \rfloor
\]
so that
\[
2^{-n-1}r < t \le 2^{-n}r.
\]
Then, by \eqref{eq1310mon} and \eqref{eq1311mon},
\[
\sup_{\Omega\cap B_t}\left(1- \frac{\hat u_r}{W}\right)  \leq \, \sup_{B_{2^{-n}r}}\,(1-\hat v_n)  \le  \exp\left\{-\gamma\sum_{k=0}^{n}\capacity(E_{2^{-k}r},B_{2^{2-k}r})\right\}.
\]
This proves the lemma.
To prove \eqref{eq1311mon}, it is enough to show that
\begin{equation}	\label{eq1106wed}
\hat v_i \ge \hat v_{i+1}\quad\text{in }\;B_{2^{1-i}r}, \quad i=0,1,2,\ldots,
\end{equation}
since $\hat v_0= \hat u_r / W$.
We observe that $v_i W$ is characterized by
\[
v_i(x)W(x)=\inf_{w \in \mathscr{F}_i} w(x),\quad x \in B_{2^{2-i}r},
\]
where
\[
\mathscr{F}_i=\left\{ w \in \mathfrak{S}^+(B_{2^{2-i}r}):  w\ge m_{i-1} W \;\text{ in }\;B_{2^{2-i}r},\;\; w \ge W \;\text{ on }\;E_{2^{-i}r}\right\}.
\]

It follows from Lemma \ref{lem1112mon}(d) that
\[
L^* (\hat v_{i+1}W)=0 \quad \text{in }\; B_{2^{1-i}r}\setminus E_{2^{-i-1}r}.
\]
Moreover, for $w\in\mathscr F_i$, we have the following observations:
\begin{enumerate}[label=(\roman*)]
\item
For all $x\in E_{2^{-i-1}r}$,
\[
\hat v_{i+1}(x)W(x) \le W(x) \le w(x).
\]
\item
If $x \in \partial E_{2^{-i-1}r}$, then
\[
\limsup_{y\to x, \,y \in B_{2^{1-i}r} \setminus E_{2^{-i-1}r}} \hat v_{i+1}(y)W(y)
\le W(x) \le w(x) \le 
\liminf_{y\to x,\, y\in B_{2^{1-i}r} \setminus E_{2^{-i-1}r}}  w(y).
\]
\item
If $x \in \partial B_{2^{1-i}r}$, then, since $\hat u_{2^{-i-1}r}$ vanishes on $\partial B_{2^{1-i}r}$, the identity
\[
\hat v_{i+1} W- m_i W = (1-m_i)\hat u_{2^{-i-1}r}
\]
gives
\[
\limsup_{y\to x,\, y \in B_{2^{1-i}r} \setminus E_{2^{-i-1}r}} \hat v_{i+1}(y) W(y)
= m_i W(x) \le
\liminf_{y \to x,\, y \in B_{2^{1-i}r} \setminus E_{2^{-i-1}r}}  w(y).
\]
\end{enumerate}

Thus, by the comparison principle,
\[
\hat v_{i+1}W \le w\quad \text{in }\;B_{2^{1-i}r} \setminus E_{2^{-i-1}r}.
\]
Since the same inequality also holds on $E_{2^{-i-1}r}$, we have
\[
\hat v_{i+1}W \le w \quad \text{in }\;B_{2^{1-i}r}.	
\]
Taking the infimum over all $w \in \mathscr F_i$ and then passing to the lower semicontinuous regularization, we obtain
\[
\hat v_{i+1}W\le \hat v_i W \;\text{ in }\;B_{2^{1-i}r}.
\]
Since $W>0$, this proves \eqref{eq1106wed} and completes the proof of the lemma.	
\end{proof}

Lemma~\ref{lem1415mon} implies that
\[
\liminf_{x \rightarrow x_0} \hat u_r(x)=W(x_0)
\]
for all sufficiently small $r>0$.
Consequently, by Lemma \ref{lem2334sun}(c), we have
\[
\hat u_r(x_0)=W(x_0).
\]
Therefore, Lemma~\ref{lem_charact2d} implies that $x_0$ is regular.

This completes the proof of sufficiency, and hence the proof of the theorem.
\qedhere
\end{proof}

The following theorems are immediate consequences of Theorems \ref{thm_wiener2d} and \ref{thm1641mon}.

\begin{theorem}		\label{thm1510mon}
Let $\Omega$ be a bounded domain in $\mathbb R^2$.
Assume that Conditions \ref{cond1} and \ref{cond2} hold.
Then a point $x_0\in\partial\Omega$ is regular with respect to $L^*$ if and only if it is regular with respect to the Laplacian.
\end{theorem}

\begin{theorem}		\label{thm1515mon}
Let $\Omega$ be a bounded regular domain in $\mathbb R^2$.
Assume that Conditions \ref{cond1} and \ref{cond2} hold.
Then, for every $f\in C(\partial\Omega)$, the Dirichlet problem
\[
L^*u=0 \;\text{ in }\;\Omega, \quad u=f \;\text{ on }\;\partial\Omega
\]
has a unique solution $u\in C(\overline\Omega)$.
\end{theorem}



\begin{thebibliography}{m}

\bibitem{Bauman84a}
Bauman, Patricia.
\textit{Equivalence of the Green's functions for diffusion operators in $\mathbb R^n$: a counterexample}.
Proc. Amer. Math. Soc. \textbf{91}(1984), no.1, 64--68.

\bibitem{Bauman84b}
Bauman, Patricia.
\textit{Positive solutions of elliptic equations in nondivergence form and their adjoints}.
Ark. Mat. \textbf{22} (1984), no. 2, 153--173.

\bibitem{Bauman85}
Bauman, Patricia.
\textit{A Wiener test for nondivergence structure, second-order elliptic equations}.
Indiana Univ. Math. J. \textbf{34} (1985), no.4, 825--844.

\bibitem{BKRS2015}
Bogachev, Vladimir I.; Krylov, Nicolai V.; R\"ockner, Michael; Shaposhnikov, Stanislav V.
\textit{Fokker-Planck-Kolmogorov equations}.
Math. Surveys Monogr., \textbf{207}.
American Mathematical Society, Providence, RI, 2015, xii+479 pp.

\bibitem{DKK25a}
Dong, Hongjie; Kim, Dong-ha; Kim, Seick.
\textit{The Dirichlet problem for second-order elliptic equations in non-divergence form with continuous coefficients}.
Math. Ann. \textbf{392} (2025), no. 1,  573--618.

\bibitem{DKK25b}
Dong, Hongjie; Kim, Dong-ha; Kim, Seick.
\textit{The Dirichlet problem for second-order elliptic equations in non-divergence form with continuous coefficients: The two-dimensional case}.
arXiv:2504.00190 [math.AP]

\bibitem{DK21}
Dong, Hongjie; Kim, Seick.
\textit{Green's function for nondivergence elliptic operators in two dimensions}.
SIAM J. Math. Anal. \textbf{53} (2021), no. 4, 4637--4656.

\bibitem{DK17}
Dong, Hongjie; Kim, Seick.
\textit{On $C^1$, $C^2$, and weak type-$(1,1)$ estimates for linear elliptic operators}.
Comm. Partial Differential Equations \textbf{42} (2017), no. 3, 417--435.

\bibitem{DEK18}
Dong, Hongjie; Escauriaza, Luis; Kim, Seick.
\textit{On $C^1$, $C^2$, and weak type-$(1,1)$ estimates for linear elliptic operators: part II}.
Math. Ann. \textbf{370} (2018), no. 1-2, 447--489.

\bibitem{EM2017}
Escauriaza, Luis; Montaner, Santiago.
\textit{Some remarks on the $L^p$ regularity of second derivatives of solutions to non-divergence elliptic equations and the Dini condition}. Rend. Lincei Mat. Appl. \textbf{28} (2017), 49--63.

\bibitem{FS84}
Fabes, E. B.; Stroock, D. W.
\textit{The $L^p$ -integrability of Green's functions and fundamental solutions for elliptic and parabolic equations}.
Duke Math. J. \textbf{51} (1984), no. 4, 997--1016.

\bibitem{GW82}
Gr\"uter, Michael; Widman, Kjell-Ove.
\textit{The Green function for uniformly elliptic equations}.
Manuscripta Math. \textbf{37} (1982), no. 3, 303--342.

\bibitem{GK24}
Gy\"ongy, Istvan; Kim, Seick.
\textit{Harnack inequality for parabolic equations in double-divergence form with singular lower order coefficients}.
J. Differential Equations \textbf{412} (2024), 857--880.

\bibitem{Herve}
Herv\'e, Rose-Marie.
\textit{Recherches axiomatiques sur la th\'eorie des fonctions surharmoniques et du potentiel}. (French)
Ann. Inst. Fourier (Grenoble) \textbf{12} (1962), 415--571.

\bibitem{HK20}
Hwang, Sukjung; Kim, Seick.
\textit{Green's function for second order elliptic equations in non-divergence form}.
Potential Anal. \textbf{52} (2022), no. 1, 27--39.

\bibitem{Kim2023}
Kim, Seick.
\textit{Recent Progress on second-order elliptic and parabolic equations in double divergence form}.
Japan-Korea developments in pure and applied mathematics, 185--211.
Springer Proc. Math. Stat., \textbf{546}.
Springer, Singapore, 2026

\bibitem{Krylov67}
Krylov, N. V.
\textit{The first boundary value problem for elliptic equations of second order}. (Russian)
Differencial'nye Uravnenija \textbf{3} (1967), 315--326.

\bibitem{Krylov2023e}
Krylov, N. V.
\textit{Elliptic equations in Sobolev spaces with Morrey drift and the zeroth-order coefficients}.
Trans. Amer. Math. Soc. \textbf{376} (2023), no.10, 7329--7351.

\bibitem{KY17}
Krylov, N. V.; Yastrzhembskiy, Timur.
\textit{On nonequivalence of regular boundary points for second-order elliptic operators}.
Comm. Partial Differential Equations \textbf{42} (2017), no.3,  366--387.

\bibitem{LSW63}
Littman, W.; Stampacchia, G.; Weinberger, H. F.
\textit{Regular points for elliptic equations with discontinuous coefficients}.
Ann. Scuola Norm. Sup. Pisa Cl. Sci. (3) \textbf{17} (1963), 43--77.

\bibitem{Mamedov}
Mamedov, F. I.
\textit{On the Harnack inequality for an equation that is formally conjugate to a linear elliptic differential equation}.
Sibirsk. Mat. Zh. \textbf{33} (1992), no. 5, 100--106, 222; translation in
Siberian Math. J. \textbf{33} (1992), no. 5, 835--841


\bibitem{MMcO2007}
Maz'ya, Vladimir; McOwen, Robert.
\textit{Asymptotics for solutions of elliptic equations in double divergence form}.
Comm. Partial Differential Equations \textbf{32} (2007), no. 1-3, 191--207.

\bibitem{Miller}
Miller, Keith.
\textit{Nonequivalence of regular boundary points for the Laplace and nondivergence equations, even with continuous coefficients}.
Ann. Scuola Norm. Sup. Pisa Cl. Sci. (3) \textbf{24} (1970), 159--163.

\bibitem{Oleinik}
Oleinik, O. A.
\textit{On the Dirichlet problem for equations of elliptic type}. (Russian)
Mat. Sbornik N.S. \textbf{24/66} (1949), 3--14.

\bibitem{Sjogren73}
Sj\"ogren, Peter.
\textit{On the adjoint of an elliptic linear differential operator and its potential theory}.
Ark. Mat. \textbf{11} (1973), 153--165.

\bibitem{Sjogren75}
Sj\"ogren, Peter.
\textit{Harmonic spaces associated with adjoints of linear elliptic operators}.
Ann. Inst. Fourier (Grenoble) \textbf{25} (1975), no. 3-4, xviii, 509--518.

\end{thebibliography}
\end{document}